\documentclass[11pt,twoside]{amsart}
\usepackage[dvipsnames]{xcolor}
\usepackage{hyperref}
\hypersetup{colorlinks=true, linkcolor=Blue, citecolor=Blue}
\usepackage{amsthm, amsmath, amscd, amssymb,centernot,txfonts}
\usepackage{tikz-cd}
\usepackage{lipsum}
\usepackage{cancel} 
\usepackage{multicol}
\usepackage{amsfonts}
\usepackage{mathrsfs}
\usepackage[all]{xy}
\usepackage[left=2.5cm,top=2.5cm,bottom=3cm,right=2.5cm]{geometry}
\usepackage{dirtytalk}
\usepackage{mathtools}
\usetikzlibrary{shapes,arrows}
\usepackage{etoolbox}
\usepackage{dynkin-diagrams}
\usepackage{enumitem}
\usepackage{chngpage}
\usepackage{adjustbox}
\usepackage[normalem]{ulem}
\usepackage[utf8]{inputenc}
\usepackage{fourier} 
\usepackage{array}
\usepackage{makecell}
\usepackage{comment}
\usepackage{cancel}
\usepackage{cleveref}
\newcommand{\longsquiggly}{\xymatrix{{}\ar@{~>}[r]&{}}}

\theoremstyle{plain}
\numberwithin{equation}{section}
\newtheorem{theorem}{Theorem}[section]
\newtheorem{proposition}[theorem]{Proposition}
\newtheorem{lemma}[theorem]{Lemma}
\newtheorem{corollary}[theorem]{Corollary}

\newtheorem{claim}[theorem]{Claim}

\theoremstyle{definition}
\newtheorem{remark}[theorem]{Remark}

\newtheorem{example}[theorem]{Example}

\newtheorem{set-up}[theorem]{Set-up}
\newtheorem{definition}[theorem]{Definition}

\newcommand*{\QEDA}{\hfill\ensuremath{\blacksquare}}
\newcommand*{\QEDB}{\hfill\ensuremath{\square}}

\tikzstyle{decision} = [diamond, draw, , 
    text width=4.5em, text badly centered, node distance=3cm, inner sep=0pt]
\tikzstyle{block} = [rectangle, draw, , 
    text width=10em, text centered, rounded corners, minimum height=2em]
\tikzstyle{block1} = [rectangle, draw, , 
    text width=5em, text centered, rounded corners, minimum height=2em]    
\tikzstyle{line} = [draw, -latex']
\tikzstyle{cloud} = [draw, ellipse,, node distance=3cm,
    minimum height=2em]
\begin{document}

\title[Extendability of general $K3$ surfaces without Gaussian maps ]{Extendability of general $K3$ surfaces without Gaussian maps and classification of non-prime Fano threefolds}

\author[P. Bangere]{Purnaprajna Bangere}
\address{Department of Mathematics, University of Kansas, Lawrence, USA}
\email{purna@ku.edu}

%\author[F.J. Gallego]{Francisco Javier Gallego}
%\address{Departamento de \'Algebra, Geometr\'ia y Topolog\'ia and Instituto de Matem\'atica Interdisciplinar,
%Universidad Complutense de Madrid, Spain}
%\email{gallego@mat.ucm.es}

\author[J. Mukherjee]{Jayan Mukherjee}
\address{Department of Mathematics, Oklahoma State University, Stillwater, USA}
\email{jayan.mukherjee@okstate.edu}


%\subjclass[2020]{14B05, 14B10, 14D06, 14D15, 14D20, 14E30, 14J10, 14J45}
%\keywords{compactification of moduli spaces, simple normal crossings, semi-log-canonical singularities, smoothings, deformations of morphisms, multiple structures, Fano varieties, Calabi-Yau varieties, varieties of general type, moduli of varieties of general type.}

\maketitle
\begin{abstract}

In \cite{BM24}, we introduced an approach to the question of extendability of projective varieties via degeneration to ribbons. In this article we build on these methods to give a new proof of optimal results on the extendability of general non-prime $K3$ surfaces, classification of non-prime \textcolor{black}{Fano threefolds and} Mukai varieties and the irreducibility of their Hilbert schemes. 
\textcolor{black}{The methods in this article also show the non-extendability of prime $K3$ surfaces for infinitely many values of $g$, for example when $g$ is of the form $g = 4k+1$, $k \geq 5$.} 
This involves degenerations of $K3$ surfaces to ribbons on embedded Hirzebruch surfaces, called $K3$ carpets. We directly give optimal upper bounds on the cohomology of the twisted normal bundle of the $K3$ carpets instead of computing coranks of Gaussian maps of the canonical curve sections as in \cite{CLM94}, \cite{CLM98}. 
As a result of independent interest, we show such $K3$ carpets also appear as degenerations of smoothable simple normal crossings of two Hirzebruch surfaces embedded by arbitrary linear series intersecting along an anticanonical elliptic curve. 
Such type II degenerations constitute a smooth locus of codimension $6$ in the Hilbert scheme of $K3$ surfaces. 

\end{abstract}
%\section{Observations to discuss with CLM and highlight in the paper}

%\begin{enumerate}
 %   \item I think we are recovering Wahl's theorem for non-prime  $K3$ surfaces. 
    %Since $\alpha(\widetilde{Y}) = 0$, we have $\alpha(X) = 0$ for a general $K3$. 
  %  In \Cref{H^0(N(-2))}, we have shown $h^0(N_{\widetilde{Y}}(-2\widetilde{H})) = 0$ in order to compute the dimension of the tangent space at the Hilbert point represented by the cone over a general $K3$. This implies $h^0(N_{X}(-2H)) = 0$ for a general non-prime $K3$ surface. Therefore $h^0(N_{C}(-H)) \geq N+1$ and hence $\alpha(C) = \operatorname{corank}(\Phi_{\omega_{C}}) \geq 1$. This is an important consequence of our method of directly computing cohomology of twists of the normal bundle on a $K3$ surface. 

   % \item The results of \cite[Lemma $4$]{CLM93} (Gaussian map of graph curve techniques), \cite{CLM94}, \cite[Proposition $2.22$]{CLM98} (Gaussian map on $K3$ surface techniques) also recovers Wahl's theorem. So why dont they highlight it ? Is the argument in item $[(1)]$ correct ?
    %They show $\operatorname{corank}(\Phi_{\omega_C}) \leq 1$ and crucially uses Wahl's theorem to to conclude that $\operatorname{corank}(\Phi_{\omega_C}) $ cannot be zero. 
%\end{enumerate}

\section{Introduction}

\begin{definition}\label{extendability}
Let $X \subset \mathbb{P}^M$ be an irreducible nondegenerate variety of codimension at least $1$. Let $k \geq 1$ be an integer. We say that $X$ is $k-$ extendable if there exists a variety $W \subset \mathbb{P}^{M+k}$
different from a cone, with dim $W =$ dim $X + k$ and having $X$ as a section by a $M-$ dimensional linear space such that $W$ is smooth along $X = W \cap \mathbb{P}^M$. We say that $X$ is precisely $k-$ extendable if it is $k-$ extendable but not $(k + 1)-$ extendable. The variety $W$ is called a $k$-extension of $X$. We say that $X$ is extendable if it is $1-$ extendable.  
\end{definition}

Extendability of a projective variety is a natural and fundamental question in projective geometry. This classical question has been the topic of intense research for decades and has revealed deep connections between the geometry of an embedding, Gaussian map of curve sections, deformations of cones over the hyperplane sections etc. We refer to \cite{Lop23} for an excellent recent survey of this topic. 
%\textcolor{red}{There is an additional motivation for the study; it gives a different proof of the classification of an important class of Fano threefolds as shown in....} \textcolor{blue}{perhaps this has to be said better, but putting it here nonetheless so as not to forget.} 

\smallskip

This article revisits the question of extendability of an embedded $K3$ surface and its \textcolor{black}{applications}  \textcolor{black}{to the} boundedness, classification and Hilbert schemes of Fano threefolds and more generally Mukai varieties, which are varieties with a canonical curve section. In \cite{CLM93}, \cite{CLM94} and \cite{CLM98},  Ciliberto, Lopez and Miranda analyze the extendability of general $K3$ surfaces by computing the coranks of Wahl maps of the canonical curve sections (see also \cite{CD22}, \cite{CD24}, \cite{K20}, \cite{Tot20}). %\textcolor{blue}{In the case of prime $K3$ surfaces the authors show that the corank of the Wahl map of the canonical curve section is one by degenerating to a union of scrolls and further degenerating to a union of planes, whose hyperplane sections are graph curves.} In the case of non-prime $K3$ surfaces, the authors first prove the surjectivity of a relevant Gaussian map on a general non-prime $K3$ surface (see \cite{CLM94}) and then conclude  \textcolor{black}{that the corank of the Wahl map of the canonical curve section is one} (see \cite{CLM98}), by invoking results from the case of prime $K3$ surfaces (see \cite[Lemma $2.3$]{CLM98} \cite[Lemma $4$]{CLM93}, see also \cite{K20} ). %\textcolor{blue}{We also refer to \cite{CD22} and \cite{CD24} for coranks of Gaussian maps of non-primitive linear systems on $K3$ surfaces.}

\smallskip

\textcolor{black}{The questions on extendability has almost always involved Gaussian map, either on the surface or going to the hyperplane curve section. As noted in \cite{Lop23} "the emerging philosophy is that if one can control the corank of the Gaussian maps of the curve section of X, then one can give information on the extendability of X." It is further noted in \cite{Lop23} that there \textcolor{black}{are} only \textcolor{black}{two} other general results (\cite{KLM11}, \cite{CDGK20}) that allows studying the extendability of surfaces, without passing to the hyperplane section.}\

\smallskip

 \textcolor{black}{The method in this article  avoids the approach via Gaussian maps altogether, either on the surface or on the hyperplane curve section. It provides another general method, via degeneration to ribbons, to study extendability} by directly computing $\alpha(X) = h^0(N_{X/\mathbb{P}^M}(-1))-M-1$ and $h^0(N_{X/\mathbb{P}^M}(-2))$ for a general $K3$ surface $X$. \textcolor{black}{Once this is accomplished,} we conclude using the following theorem due to Zak-L\'vovsky. 
%Hence this gives %This theorem gives us a criterion to determine if a projective variety $X \subset \mathbb{P}^M$ is $k-$ extendable or not.

\begin{theorem}(\cite{Lvo92}, \cite{Zak91})\label{ZL}
Let $X \subset \mathbb{P}^{M}$ be a smooth irreducible nondegenerate variety of codimension at least $1$ and suppose $X$ is not a quadric. If $\alpha(X) \leq 0$, then X is not extendable. Further, given an integer $k \geq 2$, suppose that either:
\begin{itemize}
    \item[(1)] $\alpha(X) < M$ or
    \item[(2)] $H^0(N_{X/\mathbb{P}^M}(-2)) = 0$,
\end{itemize}
If $\alpha(X) \leq k-1$, then X is not $k-$ extendable.  

\end{theorem}

By semicontinuity it is enough to give sharp upper bounds to $\alpha(X)$ by computing it over an embedded degeneration of $X$. We show that $\alpha(\widetilde{Y})$ for ribbons $\widetilde{Y} \subset \mathbb{P}^M$ called $K3$ carpets (see \Cref{ribbon}) on Hirzebruch surfaces $Y \subset \mathbb{P}^M$ embedded by the linear series \textcolor{black}{$|aC_0+bf|$ for suitably chosen values of $a$ and $b$}, achieve sharp upper bounds. Such $K3$ carpets were shown to be smoothable in \cite{BMR21} (see also \cite{GP97}).  \textcolor{black}{When $\operatorname{gcd}(a,b) > 1$, the resulting smooth fibers are non-prime $K3$ surfaces while if $\operatorname{gcd}(a,b) = 1$, the resulting smooth fibers are prime $K3$ surfaces (see \Cref{degeneration into $K3$ carpets}).} The computability of $\alpha(\widetilde{Y})$ is facilitated by the fact that $\widetilde{Y} \subset \mathbb{P}^M$ is an embedded ribbon which is a local complete intersection scheme whose reduced part is a (possibly degenerate) Hirzebruch surface $Y \subset \mathbb{P}^M$. %\sout{thickening of a (possibly degenerate) Hirzebruch surface $Y \subset \mathbb{P}^M$, where the embedding is induced by the complete linear series of a very ample line bundle $aC_0+bf$ followed by a linear embedding of projective spaces.}
We reduce the computations involving the cohomology of twists of the normal bundle $N_{\widetilde{Y}/\mathbb{P}^M}$ to computing cohomology of twists of $N_{Y/\mathbb{P}^M}$ by a series of cohomological procedures and analysis of the related coboundary maps. \par

\textcolor{black}{We now state our results. Let $\mathcal{H}_{r,g}$ denote the Hilbert scheme of $K3$ surfaces of index $r$ and genus $g$. (see \Cref{Hilbert schemes of K3 surfaces})}
For \textcolor{black}{general} $K3$ surfaces of large genus, we show:

\begin{theorem}(see \Cref{non-extendability of K3 surfaces}, \Cref{H^0(N(-2))})\label{non-extendability of K3 surfaces intro}
    Let \textcolor{black}{$X \subset \mathbb{P}^{M}$ be a general $K3$ surface in $\mathcal{H}_{r,g}$ embedded by the complete linear series of $rB$, where $B$ is a primitive very ample line bundle, with $B^2 = 2g-2$ and $M = 1+r^2(g-1)$.}
    
    \begin{enumerate}
        \item If $r \geq 5, g \geq 3$ or $r = 4, g \geq 4$ or $r = 3, g \geq 5$ or $r = 2$, $g \geq 7$, then $\alpha(X) = 0$ and consequently $X \subset \mathbb{P}^M$ is not extendable. 

       % \item[(2)] If $r = 1$, $g = 4k+1, k \geq 5$ or $r = 1$, $g = 18k+4, k \geq 1$, or $r = 1$, $g = 18k+7, k \geq 2$ or $r = 1$, $g = 18k+13, k \geq 1$ or $r = 1$, $g = 18k+16, k \geq 1$, then $\alpha(X) = 0$ and consequently $X \subset \mathbb{P}^M$ is not extendable. 
       
        \item If $r = 1$ and $g$ is either $g = 4k+1, k \geq 5$ or $g = 18k+4, k \geq 1$, or $g = 18k+7, k \geq 2$ or $g = 18k+13, k \geq 1$ or $g = 18k+16, k \geq 1$, then $\alpha(X) = 0$ and consequently $X \subset \mathbb{P}^M$ is not extendable.
        
        \item[(3)] If $r = 2, g \geq 4$ or $r \geq 3, g \geq 3$, then $H^0(N_{X/\mathbb{P}^M}(-2)) = 0$. Further, if $(r = 2, g = 3)$, then $h^0(N_{X/\mathbb{P}^M}(-2)) \leq 1$.

    \end{enumerate}
\end{theorem}

 \par

\textcolor{black}{The above theorem completely recovers optimal results in the case of non-prime $K3$ surfaces, while it recovers results for infinitely many values of $g$ in the prime case. We intend to deal with the case $r = 1$ and all possible values of $g$ more systematically in a forthcoming article. }

\vspace{0.2cm}

\color{black}

%\sout{One of the consequences of \Cref{non-extendability of K3 surfaces intro}, $(2)$, is the non-surjectivity of the Wahl map $\Phi_{\omega_C}$ in the non-prime case.}

%\begin{corollary}\label{Wahl intro}(\Cref{Wahl})
% \sout{Let $X$ be a $K3$ surface and $C \subset X$ is a smooth curve with $C \in |rB|$ with $r \geq 2$. Then the Wahl map $\Phi_{\omega_C}$ is not surjective.}
%\end{corollary}

%\sout{The result was first proved by Wahl (\cite{W87}) and later by Beauville-Merindol (\cite{BM87}) and Ciliberto-Miranda (\cite{CM90}) (see also \cite{V92}).}

\color{black}

For non-prime $K3$ surfaces of lower genus, we show 

\begin{theorem}(see \Cref{ZL for lower genus $K3$ surfaces})\label{ZL for lower genus $K3$ surfaces intro}

Let \textcolor{black}{$X \subset \mathbb{P}^{M}$ be a general $K3$ surface in $\mathcal{H}_{r,g}$ embedded by the complete linear series of $rB$, where $B$ is a primitive very ample line bundle, with $B^2 = 2g-2$ and $M = 1+r^2(g-1)$.}

    \begin{enumerate}
       % \item \textcolor{blue}{If $r = 2$, $g = 2$, then $\alpha(X) = 14$} 
        \item If $r = 2$, $g = 3$, then $\alpha(X) = 10$ 
        %(\textcolor{blue}{Needs improvement by $1$.})
        \item If $r = 2$, $g = 4$, then $\alpha(X) = 6$ 
        \item If $r = 2$, $g = 5$, then $\alpha(X) = 3$ 
        \item If $r = 2$, $g = 6$, then $\alpha(X) = 1$ 
       % \item If $r = 3$, $g = 2$, then $\alpha(X) = 10$ 
        \item If $r = 3$, $g = 3$, then $\alpha(X) \leq 4$
        \item If $r = 3$, $g = 4$, then $\alpha(X) = 1$ 
        \item If $r = 4$, $g = 3$, then $\alpha(X) = 1$

    \end{enumerate}
    
\end{theorem}

One of the key technical results proved in \cite{CLM98}, that crucially depends on the corank of the Wahl map, concerns the dimension of the component of the Hilbert scheme at points representing cones over canonical curves. Using \Cref{non-extendability of K3 surfaces intro} $(2)$ and the bounds of $\alpha(X)$ in \Cref{ZL for lower genus $K3$ surfaces intro}, we prove this result directly for the cone over a $K3$ surface. \par
%\sout{without going to Gaussian maps on canonical curves}. 
The following theorem, which is a consequence of our results and a general principle outlined in \cite{CLM98}, gives a classification of non-prime Fano-threefolds and more generally non-prime Mukai varieties of Picard rank one 
 
\begin{theorem}(see \Cref{classification of fanos with index two}, \Cref{classification of Mukai varieties with index two})\label{classification of mukai varieties intro}
\begin{enumerate}
    \item Let \textcolor{black}{$\mathcal{V}_{r,g,1}$ denote the closure of the locus of points representing smooth Fano threefolds with Picard rank one in the Hilbert scheme of smooth anticanonically embedded Fano threefolds of index $r$ and genus $g$. Let $r \geq 2$ and $g \geq 3$ (see \Cref{Hilbert scheme of Fanos and K3})}.
    Then

\begin{enumerate}
    \item $\mathcal{V}_{r,g,1} = \varnothing$ if $(r = 2, g \geq 7)$, $(r = 3, g = 3)$, $(r = 3, g \geq 5)$, $(r = 4, g \geq 4)$ or $(r \geq 5, g \geq 3)$. 
    \item For $(r = 2, 3 \leq g \leq 6), (r = 3, g = 4), (r = 4, g = 3)$, $\mathcal{V}_{r,g,1}$ is irreducible, and the families of Fano and Iskovskih form an open dense irreducible subset of $\mathcal{V}_{r,g,1}$. The general fiber of the projection map $p: \mathcal{F}_{r,g} \to \mathcal{H}_{r,g}$ as defined in \Cref{Hilbert scheme of Fanos and K3} is irreducible. Upto projective transformations, a general $K3$ surface in $\mathcal{H}_{2,6}, \mathcal{H}_{3,4}$ or $ \mathcal{H}_{4,3}$ is contained in a unique Fano threefold in $\mathcal{V}_{2,6,1}, \mathcal{V}_{3,4,1}$ or $ \mathcal{V}_{4,3,1}$ respectively.
\end{enumerate}

\item \textcolor{black}{For $n \geq 4$, let $\mathcal{V}_{n,r,g,1}$ denote the closure of the locus of smooth Fano $n-$ folds of Picard rank one in the Hilbert scheme of smooth embedded Fano varieties of dimension $n \geq 4$, index $r(n-2)$ and genus $g$ with a canonical curve section. (see \Cref{Mukai varieties}). Let $r \geq 2$ and $g \geq 3$}.
Then 
\begin{enumerate}
    \item $\mathcal{V}_{n,r,g,1} = \varnothing$ if $(r = 2, 3 \leq g \leq 4)$, $(n \geq 6, r = 2, g = 5)$, $(r = 2, g \geq 6)$, $(r = 3, g \geq 3)$, or $(r = 4, g \geq 3)$ . 
    \item $\mathcal{V}_{4,2,5,1}$ and $\mathcal{V}_{5,2,5,1}$ are irreducible, and the families of Mukai form an open dense irreducible subset of the components.
\end{enumerate}

\end{enumerate} 

\end{theorem}

Finally in \Cref{Hirzebruch main}, we show that the $K3$ carpets $\widetilde{Y} \subset \mathbb{P}^M$ %\sout{extending the possibly degenerate embeddings $Y \subset \mathbb{P}^M$, where  $Y$ is a} 
on a Hirzebruch surface $Y=\mathbb{F}_e$, $0 \leq e \leq 1$, \textcolor{black}{embedded by arbitrary linear series} 
%\sout{and the embedding is induced by the complete linear series of a very ample line bundle $aC_0+bf$ followed by a linear embedding of projective spaces,} 
lie in the closure of simple normal crossing varieties of the form $V = Y_1 \bigcup_E Y_2 \subset \mathbb{P}^M$, where each $Y_i \cong Y$ is embedded by $aC_0+bf$ and $E$ is an anticanonical elliptic curve.  The relation of these embedded type two degenerations of $K3$ surfaces \textcolor{black}{to the present context is} as follows: In \cite{CLM93}, to deal with the prime case, the authors degenerated $K3$ surfaces into a union of scrolls (this is the case $a=1$ above.) These were further degenerated to a union of planes whose hyperplane sections are graph curves with corank one Gaussian map. \textcolor{black}{The techniques introduced in this article show that in the non-prime case, degenerating the $K3$ surfaces to union of Hirzeburch surfaces embedded by arbitrary linear series, and further degenerating to $K3$ carpets, suffices to prove the results.} %\sout{requiring no further degeneration to planes and taking hyperplace curve sections or study Gaussian maps.} 
%\sout{present context, to deal with the non-prime case, we essentially, degenerate $K3$ surfaces to a union of Hirzebruch surfaces as non scrolls meeting along an elliptic curve, but instead of further degenerating to union of planes, we degenerate it to a $K3$ carpet.} 
%Comment:\textcolor{blue}{In fact I was thinking about rewriting this entire paragraph before we put it in the arxiv. We should not get into this mess as it was done here before...it would confuse the reader endlessly. But I feel the above makes it clearer. What do you think?} 

%On the other hand, J. Wahl, introduced the following map which is now called a Wahl-map

%\begin{definition}\label{Wahl map}
 %Let $C$ be a smooth irreducible curve, then the Wahl map of $C$ is the map
 %$$\Phi_{\omega_C}: \wedge^2H^0(\omega_C) \to H^0(\omega_C^{\otimes 3})$$
 %given by $\Phi_{\omega_C}(fdz \wedge gdz) = (fg'-gf')dz^3 $
%\end{definition}

%The relation of Wahl map to extendability is given by the fact that if $C$ is a curve of genus $g \geq 3$, in its canonical embedding, then $\alpha(C) = \operatorname{cork}(\Phi_{\omega_C})$. Further if $C$ is a linearly normal curve section of a non-degenerate variety $X \subset \mathbb{P}^N$ of dimension $n$ and $H^0(N_C(-2)) = 0$, then $\alpha(X) \leq \alpha(C)-n+1$. 

%In this article, we reprove the results on extendability of general $K3$ surfaces in the non-prime case, i.e, when the hyperplane section of the embedding is a multiple of the generator of the Picard group. $K3$ surface is embedded 
\vspace{0.2cm}

\par

Using \Cref{non-extendability of K3 surfaces intro} along with results of \cite{CDS20}, \cite{ABS17}, we conclude that 
\begin{corollary}\label{corank one intro}(see \Cref{corank one})
In the situation of \Cref{non-extendability of K3 surfaces intro}, if $C \in |rB|$ is any smooth hyperplane section of $X \subset \mathbb{P}^M$, then $\alpha(C) = \operatorname{cork}(\Phi_{\omega_C}) = 1$.   
\end{corollary}
The corank one theorem can also be derived from our methods  independently by showing the injectivity of the map $H^1(N_{\widetilde{Y}/\mathbb{P}^M}(-2\widetilde{H})) \to H^1(N_{\widetilde{Y}/\mathbb{P}^M}(-\widetilde{H}))$ on the ribbon $\widetilde{Y}$ (see \Cref{independent corank one} ). For this, along with the methods introduced here, one needs to use results on Gaussian maps on Hirzebruch surfaces in \cite{DM92} and \cite{DP12}. Since showing injectivity of this map is essential to \textcolor{black}{systematically} handle the case $r = 1$, i.e, the case of prime $K3$ surfaces, we will come back to this in a forthcoming article (\cite{BM25}) . 

\vspace{0.2cm}

\noindent\textbf{Acknowledgements.} We thank Professor Rick Miranda and Professor Angelo Lopez for useful discussions. We also thank Professor Thomas Dedieu for reading the paper carefully and for his helpful suggestions; in particular for pointing out that our results hold in greater generality than what was stated in our previous version. 

\section{Preliminaries on degeneration of $K3$ surfaces into $K3$ double structures on Hirzebruch surfaces}

We recall some definitions and a theorem on some special non-reduced degenerations of $K3$ surfaces in a projective space.

\begin{definition}\label{ribbon}
Let $Y$ be a reduced connected projective variety.
 \begin{enumerate}
     \item A scheme $\widetilde{Y}$ is called a ribbon on $Y$ with conormal bundle $L$ if $\widetilde{Y}_{\operatorname{red}} = Y$, $\mathcal{I}_{Y/\widetilde{Y}}^2 = (0)$ and $\mathcal{I}_{Y/\widetilde{Y}} \cong L$ as an $\mathcal{O}_Y$ module where $\mathcal{I}_{Y/\widetilde{Y}}$ is the ideal sheaf of $Y$ inside $\widetilde{Y}$ and $L$ is a line bundle on $Y$. 

     \item If $Y$ is a smooth surface, then a ribbon $\widetilde{Y}$ on $Y$ is called a $K3$ carpet if the dualizing sheaf $K_{\widetilde{Y}} \cong \mathcal{O}_{\widetilde{Y}}$ and $H^1(\mathcal{O}_{\widetilde{Y}})= 0$

\end{enumerate}
\end{definition}  

The following lemma characterizes the conormal bundle of a $K3$ carpet $\widetilde{Y}$ on a smooth surface $Y$.

\begin{lemma}(see \cite[Proposition $2.1$]{GGP08})
    A ribbon $\widetilde{Y}$ on smooth surface $Y$ with conormal bundle $L$ is a $K3$ carpet if and only if $L \cong K_Y$.
\end{lemma}

\color{black}
We recall the definition of the Hilbert scheme of $K3$ surfaces.

\begin{definition}\label{Hilbert schemes of K3 surfaces}
    Let $\mathcal{H}_{r,g}$ denote the Hilbert scheme of $K3$ surfaces of index $r$ and genus $g$, i.e, $K3$ surfaces $X \subset \mathbb{P}^{1+r^2(g-1)}$ embedded by the complete linear series of $rB$ where $B$ is a primitive very ample line bundle with $B^2 = 2g-2$. If $r = 1$, we call $\mathcal{H}_{r,g}$ the Hilbert scheme of prime $K3$ surfaces while if $r \geq 2$, $\mathcal{H}_{r,g}$ is called the Hilbert scheme of non-prime $K3$ surfaces.
\end{definition}
\color{black}

Next we recall a degeneration theorem of $K3$ surfaces into $K3$ carpets on Hirzebruch surfaces.

\begin{theorem}\label{degeneration into $K3$ carpets}(\cite[Theorem $4.1$]{BMR21} and \cite[Theorem $2.9$]{GP97})
Let $Y = \mathbb{F}_e$ be a Hirzebruch surface. Let $H = aC_0+bf$ be a very ample line bundle on $Y$, where $C_0$ is the class of section, satisfying $C_0^2 = -e$ and $f$ is the class of a fiber satisfying $f^2 = 0$ (this happens if and only if $b \geq ae+1$). Consider the embedding $Y \hookrightarrow \mathbb{P}^N$ induced by the complete linear series $|H|$.  
\begin{enumerate}
    \item Then there exists $K3$ carpets $\widetilde{Y}$ on $Y$ along with a very ample line bundle $\widetilde{H}$ on $\widetilde{Y}$ such that $\widetilde{H}|_Y = H$.

    \item If $\widetilde{Y} \hookrightarrow \mathbb{P}^M$ denote the embedding induced by the complete linear series $|\widetilde{H}|$, then $\widetilde{Y}$ is smoothable in $\mathbb{P}^M$ into smooth $K3$ surfaces embedded by the complete linear series of a very ample line bundle.

    \item For $0 \leq e \leq 2$, the Hilbert point of $\widetilde{Y}$ is smooth and the general smooth surface of the unique Hilbert component containing $\widetilde{Y}$ is a 
    \begin{itemize}
        \item[(a)] prime $K3$ surface if $\textrm{gcd}(a,b) = 1$ 
        \item[(b)] a non-prime $K3$ surface embedded by $rB$, with $r = \operatorname{gcd}(a,b)$ and $B$ is a primitive very ample line bundle. 
    \end{itemize}
\end{enumerate}
\end{theorem}

\begin{proof} Part $(1)$, $(2)$ and the statement on smoothness of Hilbert point in part $(3)$ of the Theorem follows from \cite[Theorem $3.1 (3)$, Theorem $4.3$, Theorem $5.1$]{BMR21} and \cite[Proposition $1.7$, Theorem $3.5$, Theorem $4.1$]{GP97}.

We prove part $(3)$. Let $Y = \mathbb{F}_e$ with $0 \leq e \leq 2$. Consider the $K3$ double cover of $\pi:X \to Y$ branched along $D \in |-2K_Y|$. Let $H_0 = \pi^*(\mathcal{O}_{Y}(1)) = \pi^*(aC_0+bf)$ and consider the composition  
\[
\begin{tikzcd}
X \arrow[dr, "\varphi"] \arrow[d, "\pi"] & \\
Y \arrow[r, hook] & \mathbb{P}^M
\end{tikzcd}
\]
Recall the smoothing fibers $X_t \xhookrightarrow{\varphi_t} \mathbb{P}^M$ are obtained as the image of an embedding $\varphi_t$ which is a general deformation of $\varphi$ along a one parameter family $T$ and are hence embedded inside $\mathbb{P}^M$ by the complete linear series of a line bundle $H_t$ which is a deformation of $H_0$. Since the multiplicity of a polarization remains constant on a smooth family, we have to find the multiplicity of $H_0$ on the double cover $X$. Now since $D$ is smooth and irreducible part $(3)$ for $Y = \mathbb{F}_e$, $0 \leq e \leq 2$ follows from the fact that $\textrm{Pic}(X)$ is generated by $\pi^*\mathcal{O}_Y(C_0)$ and $\pi^*\mathcal{O}_Y(f)$ (see \cite[Proposition $6.3$]{AHL10}).

\end{proof}

\section{Extendability of general $K3$ surfaces}

As mentioned before, to give an upper bound to $\alpha(X)$ for a $K3$ surface $X \subset \mathbb{P}^M$, we will degenerate $X \subset \mathbb{P}^M$ to a $K3$ carpet $\widetilde{Y} \subset \mathbb{P}^M$, using \Cref{degeneration into $K3$ carpets} and compute $\alpha(\widetilde{Y})$.  The main tool for all our computations that follow consist of the three exact sequences \Cref{1}, \Cref{2} and \Cref{3}  that decomposes the cohomology of the twisted normal bundle $N_{\widetilde{Y}/\mathbb{P}^M}(-\widetilde{H})$ of a ribbon $\widetilde{Y} \subset \mathbb{P}^M$ in terms of cohomology of various twists of the normal bundle $N_{Y/\mathbb{P}^M}$ of the reduced part.

%\begin{theorem}\label{ZL for snc}
% Let $V = Y_1 \bigcup_E Y_2 \hookrightarrow \mathbb{P}^M$ be a simple normal crossing variety of dimension $d > 1$, embedded by the complete linear series of a very ample line bundle $\widetilde{H}$, where $Y_1 = \mathbb{F}_e \hookrightarrow \mathbb{P}^N$, $e \leq 2$ is embedded by the complete linear series of a very ample line bundle $|aC_0+bf|$, $E \in |-K_{Y_1}|$, $Y_2$ is a general element of $\mathcal{F}_{E,Y}$ and $M = N + h^0((a-2)C_0+(b-e-2)f)$ (i.e, we are in the situation of Theorem \ref{Hirzebruch main}, $(1)$ (b). Let $H = \widetilde{H}|_{Y_1}$ and assume $H^1(-H+K_{Y_1}) = 0$. Let

% $$\beta = h^1(T_{Y_1}(-H+K_{Y_1})) + h^1(T_{Y_1}(-H)) - h^0(T_{Y_1}(-H)) +  h^1(-H-K_{Y_1}) + h^0(-H-2K_{Y_1})$$ 

\begin{theorem}\label{ZL for K3 carpets}
 Let $\widetilde{Y} \hookrightarrow \mathbb{P}^M$ be a $K3$ carpet embedded by the complete linear series of a very ample line bundle $\widetilde{H}$ such that the possibly degenerate embedding $Y = \mathbb{F}_e \hookrightarrow \mathbb{P}^M$ induced on the reduced part is given by the complete linear series of a very ample line bundle $H = aC_0+bf$ followed by a linear embedding of projective spaces. Assume $H^1(-H+K_Y) = 0$. 
 
 \begin{enumerate}
 \item Let $$\beta = h^1(T_Y(-H+K_Y)) + h^1(T_Y(-H)) - h^0(T_Y(-H)) +  h^1(-H-K_Y) + h^0(-H-2K_Y)$$ 
Then 
$$\alpha(\widetilde{Y}) = h^0(N_{\widetilde{Y}/\mathbb{P}^M}(-\widetilde{H}))-M-1 \leq  \beta$$ 
In particular if $\beta = 0$, then $\alpha(\widetilde{Y}) \leq 0$. Consequently, the general smooth $K3$ surfaces in the Hilbert component containing $\widetilde{Y}$ are not extendable .

\item Let $a = 2$, $\widetilde{Y}$ a general embedded $K3$ carpet represented by a general no-where vanishing section in $H^0(N_{Y/\mathbb{P}^M} \otimes K_Y)$ and set
%and $$\gamma = h^0(N_{Y/\mathbb{P}^M}(-H+K_Y)+h^0(N_{Y/\mathbb{P}^M}(-H)+h^0(-H-2K_Y)$$
\begin{align*}
    \gamma & =  h^0(N_{Y/\mathbb{P}^M}(-H+K_Y)+h^0(N_{Y/\mathbb{P}^M}(-H)+h^0(-H-2K_Y)-M-1 \\
    & \leq h^1(T_Y(-H+K_Y))+h^0(N_{Y/\mathbb{P}^M}(-H)+h^0(-H-2K_Y)-M-1
\end{align*}

Then 
$$\alpha(\widetilde{Y}) = h^0(N_{\widetilde{Y}/\mathbb{P}^M}(-\widetilde{H}))-M-1 \leq  \gamma$$ 
In particular if $\gamma = 0$, then $\alpha(\widetilde{Y}) \leq 0$. Consequently, the general smooth $K3$ surfaces in the Hilbert component containing $\widetilde{Y}$ are not extendable.

\end{enumerate}
\end{theorem}

\begin{proof}
    Part $(1)$ has been proved in \cite[Theorem $2.2$]{BM24}. We prove part $(2)$. We have the following exact sequences (see \cite[Lemma $4.1$]{GGP08})
\begin{equation}\label{1}
    0 \to N_{\widetilde{Y}/\mathbb{P}^M} (K_Y) \to N_{\widetilde{Y}/\mathbb{P}^M} \to N_{\widetilde{Y}/\mathbb{P}^M} \otimes \mathcal{O}_Y \to 0
\end{equation}
\begin{equation}\label{2}
    0 \to \mathcal{H}om(I_{\widetilde{Y}}/I_{Y}^2, \mathcal{O}_Y) \to N_{\widetilde{Y}/\mathbb{P}^M} \otimes \mathcal{O}_Y \to \mathcal{O}_Y(-2K_Y) \to 0 
\end{equation}
\begin{equation}\label{3}
    0 \to \mathcal{O}_Y(-K_Y) \to N_{Y/\mathbb{P}^M} \to \mathcal{H}om(I_{\widetilde{Y}}/I_{Y}^2, \mathcal{O}_Y) \to 0
\end{equation}

Tensoring equation \ref{1} by $\mathcal{O}_{\widetilde{Y}}(-\widetilde{H})$, we have that 

\begin{equation}\label{4}
 h^0(N_{\widetilde{Y}/\mathbb{P}^M}(-\widetilde{H})) \leq h^0(N_{\widetilde{Y}/\mathbb{P}^M}(-H+K_Y)) + h^0(N_{\widetilde{Y}/\mathbb{P}^M}(-H))
\end{equation}

We compute $h^0(N_{\widetilde{Y}/\mathbb{P}^M}(-H+K_Y))$. Tensoring equation \ref{2} by $\mathcal{O}_Y(-H+K_Y)$ we have 

\begin{equation*}
    0 \to \mathcal{H}om(I_{\widetilde{Y}}/I_{Y}^2, \mathcal{O}_Y)(-H+K_Y) \to N_{\widetilde{Y}/\mathbb{P}^M}(-H+K_Y) \to \mathcal{O}_Y(-H-K_Y) \to 0 
\end{equation*}

Hence 

\begin{equation}\label{6}
    h^0(N_{\widetilde{Y}/\mathbb{P}^M}(-H+K_Y)) \leq h^0(\mathcal{H}om(I_{\widetilde{Y}}/I_{Y}^2, \mathcal{O}_Y)(-H+K_Y)) + h^0(\mathcal{O}_Y(-H-K_Y))
\end{equation}

Tensoring equation \ref{3}, by $\mathcal{O}_Y(-H+K_Y)$ we have 

\begin{equation*}
    0 \to \mathcal{O}_Y(-H) \to N_{Y/\mathbb{P}^M}(-H+K_Y) \to \mathcal{H}om(I_{\widetilde{Y}}/I_{Y}^2, \mathcal{O}_Y)(-H+K_Y) \to 0
\end{equation*}

Since $H^1(-H) = 0$, by Kodaira vanishing theorem, and $h^0(-H) = 0$, we have that 

\begin{equation}\label{7}
   h^0(\mathcal{H}om(I_{\widetilde{Y}}/I_{Y}^2, \mathcal{O}_Y)(-H+K_Y)) = h^0(N_{Y/\mathbb{P}^M}(-H+K_Y))  
\end{equation}

We now compute $h^0(N_{\widetilde{Y}/\mathbb{P}^M}(-H))$. Tensoring equation \ref{2} by $\mathcal{O}_Y(-H)$ we have 

\begin{equation*}
    0 \to \mathcal{H}om(I_{\widetilde{Y}}/I_{Y}^2, \mathcal{O}_Y)(-H) \to N_{\widetilde{Y}/\mathbb{P}^M}(-H) \to \mathcal{O}_Y(-H-2K_Y) \to 0 
\end{equation*}

Hence 

\begin{equation}\label{7}
    h^0(N_{\widetilde{Y}/\mathbb{P}^M}(-H)) \leq h^0(\mathcal{H}om(I_{\widetilde{Y}}/I_{Y}^2, \mathcal{O}_Y)(-H)) + h^0(\mathcal{O}_Y(-H-2K_Y))
\end{equation}

Tensoring equation \ref{3}, by $\mathcal{O}_Y(-H)$ we have 

\begin{equation*}
    0 \to \mathcal{O}_Y(-H-K_Y) \to N_{Y/\mathbb{P}^M}(-H) \to \mathcal{H}om(I_{\widetilde{Y}}/I_{Y}^2, \mathcal{O}_Y)(-H) \to 0
\end{equation*}
Now we claim that 

\begin{claim}\label{splits}
    For a general ribbon $\widetilde{Y}$, the induced map $H^0(N_{Y/\mathbb{P}^M}(-H)) \to H^0(\mathcal{H}om(I_{\widetilde{Y}}/I_{Y}^2, \mathcal{O}_Y)(-H))$ is surjective.
\end{claim}

Granting the claim for now, we have that

\begin{equation}\label{7}
   h^0(\mathcal{H}om(I_{\widetilde{Y}}/I_{Y}^2, \mathcal{O}_Y)(-H)) = h^0(N_{Y/\mathbb{P}^M}(-H)) - h^0(-H-K_Y) 
\end{equation}

Hence we have that 

\begin{equation*}
 h^0(N_{\widetilde{Y}/\mathbb{P}^M}(-\widetilde{H})) \leq h^0(-H-K_Y) + h^0(N_{Y/\mathbb{P}^M}(-H+K_Y)) + h^0(N_{Y/\mathbb{P}^M}(-H)) - h^0(-H-K_Y) + h^0(-H-2K_Y)
\end{equation*}

\begin{equation}\label{8}
h^0(N_{\widetilde{Y}/\mathbb{P}^M}(-\widetilde{H})) \leq  h^0(N_{Y/\mathbb{P}^M}(-H+K_Y)) + h^0(N_{Y/\mathbb{P}^M}(-H)) + h^0(-H-2K_Y)
\end{equation}

We now compute $h^0(N_{Y/\mathbb{P}^M}(-H+K_Y))$. 
We have

\begin{equation}\label{9}
   0 \to T_Y \to T_{\mathbb{P}^M}|_Y \to N_{Y/\mathbb{P}^M} \to 0 
\end{equation}

and the Euler exact sequence restricted to $Y$

\begin{equation}\label{11}
   0 \to \mathcal{O}_Y \to \mathcal{O}_Y(H)^{M+1} \to T_{\mathbb{P}^M}|_Y \to 0 
\end{equation}

Tensoring equation \ref{11} by $\mathcal{O}_Y(-H+K_Y)$, we have 

\begin{equation*}
     H^0(\mathcal{O}_Y(K_Y))^{M+1} \to H^0(T_{\mathbb{P}^M}|_Y(-H+K_Y)) \to H^1(\mathcal{O}_Y(-H+K_Y)) 
\end{equation*}

We have that $H^0(\mathcal{O}_Y(K_Y)) = 0$ and by assumption $H^1(\mathcal{O}_Y(-H+K_Y)) = 0$ and hence $H^0(T_{\mathbb{P}^M}|_Y(-H+K_Y)) = 0$. 

%Similarly we have from the same exact sequence

%\begin{equation*}
 %   H^1(\mathcal{O}_Y(K_Y))^{N+1} \to H^1(T_{\mathbb{P}^M}|_Y(-H+K_Y)) \to H^2(\mathcal{O}_Y(-H+K_Y))
%\end{equation*}

Tensoring equation \ref{9} by $\mathcal{O}_Y(-H+K_Y)$, we have 

\begin{equation*}
    H^0(T_{\mathbb{P}^M}|_Y(-H+K_Y)) \to H^0(N_{Y/\mathbb{P}^M}(-H+K_Y)) \to H^1(T_Y(-H+K_Y))
\end{equation*}

Since $H^0(T_{\mathbb{P}^M}|_Y(-H+K_Y)) = 0$, 

\begin{equation}\label{14}
    h^0(N_{Y/\mathbb{P}^M}(-H+K_Y)) \leq h^1(T_Y(-H+K_Y))
\end{equation}

Plugging everything back into equation \ref{8}, we have

\begin{equation}\label{15}
h^0(N_{\widetilde{Y}/\mathbb{P}^M}(-\widetilde{H})) \leq  M+1 + h^1(T_Y(-H+K_Y)) + h^0(N_{Y/\mathbb{P}^M}(-H)) + h^0(-H-2K_Y)
\end{equation}

\end{proof}

We now prove the claim.

\smallskip

\textit{Proof of Claim \ref{splits}.} Consider the exact sequence given by twisting equation \ref{3} by $-H$. 
\begin{equation}\label{16}
0 \to -K_Y-H \to N_{Y/\mathbb{P}^M}(-H) \to \mathcal{H}om(I_{\widetilde{Y}}/I_{Y}^2, \mathcal{O}_Y)(-H) \to 0
\end{equation}

Let $p: Y \to \mathbb{P}^1$ denote the smooth fibration. Now note that $-K_Y-H = -(b-e-2)f$ and hence $R^1p_*(-H-K_Y) = 0$. Hence we have the following exact sequence by pushing forward the exact sequence \ref{16}

\begin{equation}\label{17}
    0 \to p_*(-H-K_Y) \to p_*(N_{Y/\mathbb{P}^M}(-H)) \to p_*(\mathcal{H}om(I_{\widetilde{Y}}/I_{Y}^2, \mathcal{O}_Y)(-H)) \to 0
\end{equation}
    
It is enough to show that the exact sequence \ref{17} is split for a general ribbon $\widetilde{Y}$. To see this we first show that there is a surjection from $p_*(N_{Y/\mathbb{P}^M}(-H))$ to $p_*(-H-K_Y)$. To see this consider the exact sequence \ref{9} twisted by $-H$ to get 

\begin{equation}\label{18}
   0 \to T_Y(-H) \to T_{\mathbb{P}^M}|_Y(-H) \to N_{Y/\mathbb{P}^M}(-H) \to 0 
\end{equation}

We show that

\begin{equation*}
   R^1p_*(T_{\mathbb{P}^M}|_Y(-H)) = 0, p_*(T_Y(-H)) = \mathcal{O}_{\mathbb{P}^1}(-(b-e)) \operatorname{and} R^1p_*(T_Y(-H)) =  \mathcal{O}_{\mathbb{P}^1}(-(b-e-2))  
\end{equation*}

To see the first vanishing once again note that from the Euler exact sequence pulled back to $Y$ and twisted by $-H$, we have 

\begin{equation}\label{24}
    0 \to \mathcal{O}_Y(-H) \to \mathcal{O}_Y^{\otimes M+1} \to T_{\mathbb{P}^M}|_Y(-H) \to 0
\end{equation}

Since $R^1p_*(\mathcal{O}_Y) = R^2(p_*\mathcal{O}_Y(-H)) = 0$, we have that $R^1p_*(T_{\mathbb{P}^M}|_Y(-H)) = 0$. Now we turn to $R^ip_*(T_Y(-H))$. Consider the exact sequence of horizontal and vertical tangent bundles on $Y$ twisted by $-H$
\begin{equation}\label{19}
    0 \to T_{Y/\mathbb{P}^1}(-H) \to T_Y(-H) \to p^*T_{\mathbb{P}^1}(-H) \to 0
\end{equation}

Using the relative Euler exact sequence we have that 
$T_{Y/\mathbb{P}^1} = 2C_0+ef$, so that $T_{Y/\mathbb{P}^1}(-H) = -(b-e)f$. Therefore $p_*(T_{Y/\mathbb{P}^1}(-H)) = \mathcal{O}_{\mathbb{P}^1}(-(b-e))$ and $R^ip_*(T_{Y/\mathbb{P}^1}(-H)) = R^ip_*(-(b-e)f) = 0$ for $i = 1,2$. Now $p^*T_{\mathbb{P}^1}(-H) = -2C_0-(b-2)f$. Therefore, $p_*p^*T_{\mathbb{P}^1}(-H) = p_*(-2C_0-(b-2)f) = \mathcal{O}_{\mathbb{P}^1}(-(b-2)) \otimes p_*(-2C_0) = 0$ and

\begin{align*}\label{20}
  R^1p_*(p^*T_{\mathbb{P}^1}(-H)) & = R^1p_*(-2C_0-(b-2)f) \\
  & = \mathcal{O}_{\mathbb{P}^1}(-(b-2)) \otimes R^1p_*(-2C_0) \\
  & = \mathcal{O}_{\mathbb{P}^1}(-(b-2)) \otimes (p_*(\mathcal{O}_Y))^* \otimes (\bigwedge^2(\mathcal{O}_{\mathbb{P}^1} \oplus \mathcal{O}_{\mathbb{P}^1}(-e)))^* \\ 
  & = \mathcal{O}_{\mathbb{P}^1}(-(b-e-2))
\end{align*}

Pushing forward the sequence \ref{19}, we have that 

\begin{equation}\label{21}
0 \to p_*(T_{Y/\mathbb{P}^1}(-H)) = \mathcal{O}_{\mathbb{P}^1}(-(b-e)) \to p_*T_Y(-H) \to 0    
\end{equation}

and 

\begin{equation}\label{22}
   0 \to R^1p_*T_Y(-H) \to \mathcal{O}_{\mathbb{P}^1}(-(b-e-2)) \to 0 
\end{equation}

%Now pushing forward the sequence \ref{19}, we have that $R^1p_*(T_Y(-H)) = \mathcal{O}_{\mathbb{P}^1}(-(b-e-2)) $ and we have an isomorphism $p_*(T_Y(-H)) = \mathcal{O}_{\mathbb{P}^1}(-(b-e))$. 
Now pushing forward the exact sequence \ref{18}, we have that 

\begin{multline}\label{23}
0 \to \mathcal{O}_{\mathbb{P}^1}(-(b-e)) = p_*T_Y(-H) = p_*(T_{Y/\mathbb{P}^1}(-H)) \to p_*(T_{\mathbb{P}^M}|_Y(-H)) \to p_*(N_{Y/\mathbb{P}^M}(-H)) \\ \xrightarrow[]{\lambda} \mathcal{O}_{\mathbb{P}^1}(-(b-e-2)) = p_*(-H-K_Y) \to 0
\end{multline}

We now end the proof of the claim by showing that for a general ribbon $\widetilde{Y}$, which is defined by an injection $0 \to -K_Y \to N_{Y/\mathbb{P}^M}$, the map $\lambda$ is a section to the injective map induced by the pushforward after twisting by $-H$. 

$$p_*(-H-K_Y) \to p_*(N_{Y/\mathbb{P}^M}(-H))$$

So the question is can we choose a ribbon or equivalently, an injection $0 \to -K_Y \to N_{Y/\mathbb{P}^M}$, such that the induced map $\beta$,  $$0 \to p_*(-H-K_Y) = \mathcal{O}_{\mathbb{P}^1}(-(b-e-2)) \xrightarrow[]{\beta} p_*(N_{Y/\mathbb{P}^M}(-H))$$

which sits in the following diagram

\begin{tikzcd}\label{section}
  &&& 0 \arrow[d] && \\
  &&& p_*(-H-K_Y) \arrow[d, "\beta"] && \\
  0 \arrow[r] & \mathcal{O}_{\mathbb{P}^1}(-(b-e)) \arrow[r] &  p_*(T_{\mathbb{P}^M}|_Y(-H)) \arrow[r, "\delta"] & p_*(N_{Y/\mathbb{P}^M}(-H)) \arrow[r, "\lambda"] & p_*(-H-K_Y) \arrow[r] & 0  
\end{tikzcd}

\smallskip 

satisfies $\lambda \circ \beta = \operatorname{id}$. It is enough to choose $\beta$ such that $\lambda \circ \beta$  is an isomorphism and since $p_*(-H-K_Y) = \mathcal{O}_{\mathbb{P}^1}(-(b-e-2))$ is a line bundle, it is enough to choose $\beta$ such that $\lambda \circ \beta$ is non-zero. This happens in particular if $\operatorname{Im}(\beta)$ is not contained in $\operatorname{Ker}(\lambda)$. So we need to show that $\beta$ does not factor as 

\[
\begin{tikzcd}
   & & p_*(-H-K_Y) \arrow[d, "\beta"] \arrow[dl] \\
   0 \arrow[r] & \operatorname{Ker}(\lambda) \arrow[r] & p_*(N_{Y/\mathbb{P}^M}(-H))
\end{tikzcd}
\]

It is enough to show that 

\begin{equation}\label{32}
    \operatorname{Hom}(p_*(-H-K_Y), \operatorname{Ker}(\lambda)) \to \operatorname{Hom}(p_*(-H-K_Y), p_*(N_{Y/\mathbb{P}^M}(-H)) 
\end{equation}

does not surject. To see this note that (by sequence \ref{23} and diagram \ref{section}),  we have an exact sequence

$$0 \to \mathcal{O}_{\mathbb{P}^1}(-(b-e)) = p_*T_{Y/\mathbb{P}^1}(-H) \to p_*(T_{\mathbb{P}^M}|_Y(-H)) \to \operatorname{Ker}(\lambda) \to 0$$

Applying the functor $\operatorname{Hom}(p_*(-H-K_Y),-)$ we get 

\begin{multline}\label{30}
0 \to \operatorname{Hom}(p_*(-H-K_Y), 
 p_*T_{Y/\mathbb{P}^1}(-H))  \to \operatorname{Hom}(p_*(-H-K_Y), p_*(T_{\mathbb{P}^M}|_Y(-H))) \\ \to \operatorname{Hom}(p_*(-H-K_Y),\operatorname{Ker}(\lambda)) \to \operatorname{Ext^1}(p_*(-H-K_Y), 
 p_*T_{Y/\mathbb{P}^1}(-H)) \\ \xrightarrow[]{\eta} \operatorname{Ext^1}(p_*(-H-K_Y), 
 p_*(T_{\mathbb{P}^M}|_Y(-H))
\end{multline}

%Since $p_*(-H-K_Y) = \mathcal{O}_{\mathbb{P}^1}(-(b-e-2))$, we have that $\operatorname{Hom}(\mathcal{O}_{\mathbb{P}^1}(-(b-e-2)), 
 %\mathcal{O}_{\mathbb{P}^1}(-(b-e))) = 0$ and $\operatorname{Ext^1}(\mathcal{O}_{\mathbb{P}^1}(-(b-e-2)), 
% \mathcal{O}_{\mathbb{P}^1}(-(b-e))) = $

We show that the map $\operatorname{Ext^1}(p_*(-H-K_Y), 
 p_*T_{Y/\mathbb{P}^1}(-H))  \xrightarrow[]{\eta} \operatorname{Ext^1}(p_*(-H-K_Y), 
 p_*(T_{\mathbb{P}^M}|_Y(-H))$ is an injective map and is in fact an isomorphism. To see this, push forward sequence \ref{24}, to get 

 \begin{equation}\label{25}
     0 \to \mathcal{O}_{\mathbb{P}^1}^{M+1} \to p_*(T_{\mathbb{P}^M}|_Y(-H)) \to R^1p_*(-H) \to 0
 \end{equation}

Note that $\operatorname{Ext^i}(p_*(-H-K_Y), \mathcal{O}_{\mathbb{P}^1}^{M+1}) = H^i(\mathcal{O}_{\mathbb{P}^1}(b-e-2))^{\oplus M+1} = 0$ for $i = 1,2$ if $b \geq e+1$. Hence applying the functor $\operatorname{Hom}(p_*(-H-K_Y), -)$ we get 

\begin{equation}\label{26}
     0 \to \operatorname{Ext^1}(p_*(-H-K_Y), 
 p_*(T_{\mathbb{P}^M}|_Y(-H))) \to \operatorname{Ext^1}(p_*(-H-K_Y), 
 R^1p_*(-H)) \to 0
 \end{equation}

Composing the map $\eta$ with the isomorphism in sequence \ref{25}, we see that it is enough to show that the following map $\theta$ 

\begin{equation}\label{27}
   \operatorname{Ext^1}(p_*(-H-K_Y), 
 p_*T_{Y/\mathbb{P}^1}(-H))  \xrightarrow[]{\theta} \operatorname{Ext^1}(p_*(-H-K_Y), 
 R^1p_*(-H)) 
\end{equation}

is an isomorphism. To see this consider the relative Euler sequence corresponding to the projective bundle $Y \to \mathbb{P}^1$, twisted by $-H$ 

\begin{equation}\label{28}
0 \to \mathcal{O}_Y(-H) \to p^*(\mathcal{O}_{\mathbb{P}^1} \oplus \mathcal{O}_{\mathbb{P}^1}(-e))^* \otimes (C_0-H) \to T_{Y/\mathbb{P}^1}(-H) \to 0
\end{equation}

Pushing forward the above exact sequence, we have that 

\begin{equation}\label{29}
 0 \to p_*T_{Y/\mathbb{P}^1}(-H) \to R^1p_*(-H) \to 0 
\end{equation}

Now the map in sequence \ref{27} is induced by the isomorphism from sequence \ref{29} is hence an isomorphism. Going back to sequence \ref{30}, we have 

\begin{multline}\label{31}
0 \to \operatorname{Hom}(p_*(-H-K_Y), 
 p_*T_{Y/\mathbb{P}^1}(-H))  \to \operatorname{Hom}(p_*(-H-K_Y), p_*(T_{\mathbb{P}^M}|_Y(-H))) \\ \to \operatorname{Hom}(p_*(-H-K_Y),\operatorname{Ker}(\lambda)) \to 0
\end{multline}

Therefore, to show that the map in \ref{32} does not surject, it is enough to show that

\begin{equation}\label{33}
  \operatorname{Hom}(p_*(-H-K_Y), p_*(T_{\mathbb{P}^M}|_Y(-H)) \to \operatorname{Hom}(p_*(-H-K_Y), p_*(N_{Y/\mathbb{P}^M}(-H))  
\end{equation}

is not surjective. Since $p^*p_*(-H-K_Y) = -H-K_Y$, by the projection formula, this is the same as showing 

\begin{equation}\label{34}
   H^0(T_{\mathbb{P}^M}|_Y(K_Y)) \to H^0(N_{Y/\mathbb{P}^M}(K_Y)) 
\end{equation}
does not surject

Now by taking the cohomology of the exact sequence 

$$0 \to T_Y \otimes K_Y \to T_{\mathbb{P}^M|_Y} \otimes K_Y \to N_{Y/\mathbb{P}^M}(K_Y) \to 0$$

we have 

$$0 \to H^0(T_Y \otimes K_Y) \to H^0(T_{\mathbb{P}^M|_Y} \otimes K_Y) \to H^0(N_{Y/\mathbb{P}^M}(K_Y)) \to H^1(T_Y \otimes K_Y) \to H^1(T_{\mathbb{P}^M|_Y} \otimes K_Y)$$

Now the required non-surjection of sequence \ref{34} follows from the fact that $H^1(T_Y \otimes K_Y) = \mathbb{C}^2$ and $H^1(T_{\mathbb{P}^M|_Y} \otimes K_Y) = \mathbb{C}$. \QEDB

In the following lemmas we find out for what values of $a, b, e$, the estimates $\beta$ and $\alpha$ in \Cref{ZL for K3 carpets} are zero. The cohomology of the line bundles in Lemmas \ref{2C_0+ef-H}, \ref{2f-H}, \ref{-H-K_Y}, \ref{-H-2K_Y} follow by pushing forward onto $\mathbb{P}^1$ after possibly using Serre duality on $Y$ (see \cite[III, Ex $8.4$]{Hartshorne}).

\begin{lemma}\label{2C_0+ef-H}
Let $H = aC_0+bf$ be a line bundle on $\mathbb{F}_e$ with $a \geq 1$. Then 
\begin{enumerate}
    \item[(1)] $H^1(2C_0+ef-H) = 0$ if 
\[
\begin{cases}
   a = 1, & b \leq 1  \\
   a = 2, & b \leq e+1 \\
   a = 3 & \\
   a \geq 4, & b \geq ae-2e+1 \\
\end{cases}
\]

\item[(2)] $H^0(2C_0+ef-H) = 0$ if 
\[
\begin{cases}
   a = 1, & b \geq e+1  \\
   a = 2, & b \geq e+1 \\
   a \geq 4 & \\
\end{cases}
\] 

\end{enumerate}

\end{lemma}

\begin{lemma}\label{2f-H}
Let $H = aC_0+bf$ be a line bundle on $\mathbb{F}_e$ with $a \geq 1$. Then 
\begin{enumerate}
    \item[(1)] $H^1(2f-H) = 0$ if 
\[
\begin{cases}
   a = 1 &   \\
   a \geq 2, & b \geq ae-e+3 \\
   
\end{cases}
\]

\item[(2)] $H^0(2f-H) = 0$

\end{enumerate}  

\end{lemma}

\begin{lemma}\label{-H-K_Y}
Let $H = aC_0+bf$ be a line bundle on $\mathbb{F}_e$ with $a \geq 1$ and $b \geq ae+1$. Then $H^1(-H+K_{F_e}) = 0$. Further, $H^1(-H-K_{F_e}) = 0$ if 
 \[
\begin{cases}
   a = 1, & b \leq 3  \\
   a = 2, & b \leq 3+e \\
   a = 3 & \\
   a \geq 4, & b \geq ae-2e+3 \\
\end{cases}
 \]

\end{lemma}

\begin{lemma}\label{-H-2K_Y}
Let $H = aC_0+bf$ be a line bundle on $\mathbb{F}_e$ with $a \geq 1$ and $b \geq ae+1$. Then $H^0(-H-2K_{\mathbb{F}_e}) = 0$ if 
\[
\begin{cases}
   a \leq 4 & b \geq 2e+5  \\
   a \geq 5 &   \\
   
\end{cases}
\]
    
\end{lemma}

\begin{lemma}\label{T_Y(-H)}
Let $H = aC_0+bf$ be a line bundle on $\mathbb{F}_e$ with $a \geq 1$ and $b \geq ae+1$. Then 
\begin{enumerate}
    \item[(1)] $H^1(T_{\mathbb{F}_e}(-H)) = 0$ if 
\[
\begin{cases}
   a = 1, & b \leq 1  \\
   a = 2, & ae-e+3 \leq b \leq e+1  \\
   a = 3, & b \geq ae-e+3 \\
   a \geq 4, & b \geq \operatorname{max}\{ae-2e+3, ae-e+3\}
   
\end{cases}
\]

  \item[(2)] $H^1(T_{\mathbb{F}_e}(-H+K_{\mathbb{F}_e})) = 0$ if 
\[
\begin{cases}
   a = 1, & b \geq e+1  \\
   a \geq 2, & b \geq ae+1 \\
\end{cases}
\] 

\item[(3)] $H^0(T_{\mathbb{F}_e}(-H)) = 0$ if 
\[
\begin{cases}
   a = 1, & b \geq e+1  \\
   a = 2, & b \geq e+1 \\
   a \geq 3 & \\
\end{cases}
\] 

\end{enumerate}

\end{lemma}

\noindent\textit{Proof.} Let $Y = \mathbb{F}_e$ and let $p: Y \to \mathbb{P}^1$ denote the projection to $\mathbb{P}^1$. Note that we have an exact sequence 
$$0 \to T_{Y/\mathbb{P}^1} \to T_Y \to p^*T_{\mathbb{P}^1} \to 0$$
and that $T_{Y/\mathbb{P}^1} = 2C_0+ef$ and $p^*T_{\mathbb{P}^1} = 2f$. So it is enough to show $H^1(2C_0+ef-H) = H^1(2f-H) = 0$. Now the claim follows by combining Lemmas \ref{2C_0+ef-H} and \ref{2f-H}. \QEDA

\begin{proposition}\label{calculating beta}
%Let $V = Y_1 \bigcup_E Y_2 \hookrightarrow \mathbb{P}^M$ be a simple normal crossing variety of dimension $d > 1$, embedded by the complete linear series of a very ample line bundle $\widetilde{H}$, where $Y_1 = \mathbb{F}_e \hookrightarrow \mathbb{P}^N$, $e \leq 2$ is embedded by the complete linear series of a very ample line bundle $|aC_0+bf|$, $E \in |-K_{Y_1}|$, $Y_2$ is a general element of $\mathcal{F}_{E,Y}$ and $M = N + h^0((a-2)C_0+(b-e-2)f)$ (i.e, we are in the situation of Theorem \ref{Hirzebruch main}, $(1)$ (b). Then $\alpha(V) = \beta = 0$ if 
Let $\widetilde{Y} \hookrightarrow \mathbb{P}^M$ be a $K3$ carpet embedded by the complete linear series of a very ample line bundle $\widetilde{H}$ such that the possibly degenerate embedding $Y = \mathbb{F}_e \hookrightarrow \mathbb{P}^M$ induced on the reduced part is given by the complete linear series of a very ample line bundle $H = aC_0+bf$ followed by a linear embedding of projective spaces.   
Then 

\begin{enumerate}
    \item $\beta = 0$ if
    \[
\begin{cases}
    a = 3, b \geq 5 & e = 0  \\
    a = 3, b \geq 7 & e = 1  \\
    a = 3, b \geq 9 & e = 2  \\
\color{black}
    a = 3, b \geq 11 & e \geq 3 \\
    a = 3, b \geq 3e+1 & e \geq 4 \\
\color{black}
    a = 4, b \geq 5 & e = 0 \\
    a \geq 5, b \geq 3 & e = 0 \\
    a \geq 4, b \geq a+2 & e = 1 \\
    a \geq 4, b \geq ae+1 & e \geq 2 \\
    
\end{cases}
  \]

  \item If $a = 2$, then $\gamma = 0$ if $b \geq 2e+5$

 % \item If $a = 1$, then $\alpha(\widetilde{Y}) = 0$ if $b \geq 2e+5$

\end{enumerate}

\end{proposition}
 
\noindent\textit{Proof $(1)$.} This follows by combining Lemma \ref{-H-K_Y}, \ref{-H-2K_Y}, \ref{T_Y(-H)} and applying Theorem \ref{ZL for K3 carpets}. \\

\noindent\textit{Proof $(2)$.} Note that by Lemma \ref{T_Y(-H)}, $(2)$, $h^1(T_Y(-H+K_Y)) = 0$ if $a = 2$ and $b \geq 2e+1$. Similarly by Lemma \ref{-H-2K_Y}, $h^0(-H-2K_Y) = 0$ if $a = 2$ and $b \geq 2e+5$. Now we show that $h^0(N_{Y/\mathbb{P}^M}(-H)) = M+1$. Twisting the sequence $$0 \to T_Y \to T_{\mathbb{P}^M}|_Y \to N_{Y/\mathbb{P}^M} \to 0$$ by $-H$ and taking cohomology, we have 

\begin{equation}\label{35}
0 \to H^0(T_Y(-H)) \to H^0(T_{\mathbb{P}^M}|_Y(-H)) \to H^0(N_{Y/\mathbb{P}^M}(-H)) \to H^1(T_Y(-H)) \to H^1(T_{\mathbb{P}^M}|_Y(-H))
\end{equation}

Now by Lemma \ref{T_Y(-H)}, part $(3)$, we have that $H^0(T_Y(-H)) = 0$ if $a = 2$ and $b \geq e+1$. Turning to the Euler sequence on $\mathbb{P}^M$, restricting to $Y$ twisting by $-H$ and taking cohomology, we get

\begin{equation}\label{36}
    0 \to H^0(\mathcal{O}_Y^{\oplus M+1}) \to H^0(T_{\mathbb{P}^M}|_Y(-H)) \to 0
\end{equation}

Hence $h^0(T_{\mathbb{P}^M}|_Y(-H)) = M+1$. So it is enough to show that the map $H^1(T_Y(-H)) \to H^1(T_{\mathbb{P}^M}|_Y(-H))$ is injective. Continuing with the long exact sequence of cohomology in \ref{36}, we have that 

\begin{equation}\label{36}
    0 \to H^1(T_{\mathbb{P}^M}|_Y(-H)) \to H^2(-H) \to 0
\end{equation}

Similary tensoring the exact sequence of horizontal and vertical tangent bundles by $-H$, we have

\begin{equation}\label{37}
    0 \to T_{Y/\mathbb{P}^1}(-H) \to T_Y(-H) \to p^*(T_{\mathbb{P}^1})(-H) \to 0
\end{equation}

and taking the cohomology, we have that 

\begin{equation}\label{38}
    H^0(p^*(T_{\mathbb{P}^1})(-H)) \to H^1(T_{Y/\mathbb{P}^1}(-H)) \to H^1(T_Y(-H)) \to H^1(p^*(T_{\mathbb{P}^1})(-H)) \to 0
\end{equation}

Now note that $p^*(T_{\mathbb{P}^1})(-H) = -2C_0-(b-2)f$. We have that $H^0(-2C_0-(b-2)f) = 0$ and $H^1(-2C_0-(b-2)f) = H^1((b-e-4)f)^* = 0$ if $b \geq e+3$. hence we have 

\begin{equation}\label{39}
    0 \to H^1(T_{Y/\mathbb{P}^1}(-H)) \to H^1(T_Y(-H)) \to  0
\end{equation}

Composing the map $H^1(T_Y(-H)) \to H^1(T_{\mathbb{P}^M}|_Y(-H))$ in sequence \ref{35} with the isomorphisms given by the sequences \ref{36} and \ref{39}, we see that it is enough to show that the map 

\begin{equation}\label{40}
     H^1(T_{Y/\mathbb{P}^1}(-H)) \to H^2(-H)
\end{equation}

is injective. But the above map comes as follows: Take the relative Euler sequence of the map $Y \to \mathbb{P}^1$, given in sequence \ref{28} twist by $-H$ and take the cohomology to get

\begin{equation}\label{41}
    0 \to H^1(T_{Y/\mathbb{P}^1}(-H)) \to H^2(-H)
\end{equation}

Hence we have that $\gamma = 0$ if $a = 2$ and $b \geq 2e+5$. \\

\QEDB

\subsection{Extendability of general $K3$ surfaces of higher genus}

The following theorem is the first of our main results in the article. We give upper bounds $\alpha(X)$ for a general $K3$ surface $X \subset \mathbb{P}^M$ of index $r$ and large genus $g$. Using Theorem \ref{ZL for K3 carpets}, we compute $\alpha(\widetilde{Y})$ for a suitably chosen $K3$ carpet $\widetilde{Y}$, which by semicontinuity gives us an upper bound.

\begin{theorem}\label{non-extendability of K3 surfaces}
    Let \textcolor{black}{$X \subset \mathbb{P}^{M}$ be a general $K3$ surface in $\mathcal{H}_{r,g}$ embedded by the complete linear series of $rB$, where $B$ is a primitive very ample line bundle, with $B^2 = 2g-2$ and $M = 1+r^2(g-1)$.} Assume either of the following conditions hold:
    \begin{enumerate}
        \item[(1)] $r \geq 5, g \geq 3$
        \item[(2)] $r = 4, g \geq 4$
        \item[(3)] $r = 3, g \geq 5$ 
        \item[(4)] $r = 2$, $g \geq 7$ 
        \item[(5)] $r = 1$ and $g$ is either $g = 4k+1, k \geq 5$ or $g = 18k+4, k \geq 1$, or $g = 18k+7, k \geq 2$ or $g = 18k+13, k \geq 1$ or $g = 18k+16, k \geq 1$
        %and 
                  % \begin{enumerate}
                   %    \item[(a)]  $g = 2pq+1, p \geq 2, q \geq 2$, $p$, $q$ coprime or
                    %   \item[(b)] $g = 2pq-p^2+1, p \geq 2, q \geq p+1$, $p$, $q$ coprime 
                   %\end{enumerate}

       % \item[(5)] $r = 1$, $g = 11$ or $g \geq 13$  
                   %\begin{enumerate}
                     %  \item[(a)] $g = 6q+1$, $3 \nmid q$, $q \geq 4$
                     %  \item[(b)] $g = 6q+1$, $3 \nmid q+1$, $q \geq 5$
                      % \item[(c)] $g = 8q+1$, $2 \nmid q$, $q \geq 4$
                      % \item[(d)] $g = 2pq+1$, $p, q$ coprime, $p \geq 5, q \geq 5$
                       %\item[(e)] $g = 2pq-p^2+1$, $p, q$ coprime, $p \geq 4, q \geq p+2$ 
                   %\end{enumerate}
        
    \end{enumerate}
\end{theorem}

Then $\alpha(X) = 0$ and consequently $X \subset \mathbb{P}^M$ is not extendable.

\begin{proof}
    %We first show that a general $K3$ surface of index $r$ and genus $g$ as mentioned in Theorem \ref{non-extendability of K3 surfaces} is not extendable. For this, 
    We use Proposition \ref{calculating beta}, to show that $\beta = 0$ for $a \geq 3$ and suitable values of $b$ and $\gamma = 0$ for $a = 2$ and suitable values of $b$ and hence $\alpha(\widetilde{Y}) = 0$ for a suitably chosen $K3$ carpet $\widetilde{Y}$. 
    %type II (\textcolor{black}{you are referring to Kondoa's paper. Should we indicate what Type II mean...makes it less irrritating for the reader}) normal crossing $Y_1 \bigcup Y_2$. 
    \begin{enumerate}
        \item[(1)] For $r \geq 5$, we can choose $a = r, e = 0, b = rm, m \geq 1$ to get non-extendability of general $K3$ surfaces of index $r$ and genus $2m+1$, $m \geq 1$ and $a = r, e = 1, b = rm, m \geq 2$ to get non-extendability of general $K3$ surfaces of index $r$ and genus $2m$ with $m \geq 2$. 

        \smallskip
        
        \item[(2)] For $r = 4$, we can choose $a = 4, e = 0, b = 4m, m \geq 2$ to get non-extendability of general $K3$ surfaces of index $r$ and genus $2m+1$, $m \geq 2$ and $a = 4, e = 1, b = 4m, m \geq 2$ to get non-extendability of general $K3$ surfaces of index $4$ and genus $2m$ with $m \geq 2$. 

       \smallskip
        
        \item[(3)] For $r = 3$, we can choose $a = 3, e = 0, b = 3m, m \geq 2$ to get non-extendability of general $K3$ surfaces of index $3$ and genus $2m+1$, $m \geq 1$ and $a = 3, e = 1, b = 3m, m \geq 3$ to get non-extendability of general $K3$ surfaces of index $3$ and genus $2m$ with $m \geq 2$. 
        
        \smallskip

        \item[(4)] For $r = 2$, we use part $(2)$. One can choose $a = 2, e = 0, b = 2m, m \geq 3$ to get non-extendability of general $K3$ surfaces of index $2$ and genus $2m+1$, $m \geq 3$ and $a = 2, e = 1, b = 2m, m \geq 4$ to get non-extendability of general $K3$ surfaces of index $2$ and genus $2m$ with $m \geq 4$.
       
        \item[(5)] For $r = 1$, we first use part $(2)$. We can choose $a = 2, e = 0$  and $b = k$ odd, $k \geq 5$ to get non-extendability of general prime $K3$ surfaces of genus $4k+1$ and $a = 1, e = 1$  and $b = k+1$, $k \geq 7$, $k$ even to get non-extendability of general prime $K3$ surfaces of genus $4k+1$. Then we use part $(1)$. One can choose $a = 3, e = 0, b = 3k+1, k \geq 2$ to get non-extendability of general prime $K3$ surfaces of genus $18k+7$. Similarly the rest follows by choosing $a = 3, e = 0, b = 3k+2, k \geq 1$, $a = 3, e = 1, b = 3k+i, k \geq 2$ and $i = 1,2$.

       % \item[(4)] For $r = 2$, we can choose $a = 2p, e = 0, b = 2q$ with $p, q$ coprime, $p \geq 2, q \geq 2$ to get non-extendability of general $K3$ surfaces of index $2$ and genus $2pq+1$ and $a = 2p, e = 1, b = 2q$ with $p, q$ coprime, $p \geq 2, q \geq p+1$ to get non-extendability of general $K3$ surfaces of index $2$ and genus $2pq-p^2+1$
        
      % \smallskip

      % \item[(5)]  For $r = 1$, we use part $(3)$. One can choose $a = 1, e = 0, b = m, m \geq 5$ to get non-extendability of general $K3$ surfaces of index $1$ and genus $2m+1$, $m \geq 5$ and $a = 1, e = 1, b = m, m \geq 7$ to get non-extendability of general $K3$ surfaces of index $1$ and genus $2m$ with $m \geq 7$.
        
      %  \item[(5)] For $r = 1$, we can choose $a = 3, e = 0, b = q$ with $3 \nmid q$ with $q \geq 4$ to get non-extendability of general $K3$ surfaces of index $1$ and genus $6q+1$ or $a = 3, e = 0, b = q$ with $3 \nmid q$ with $q \geq 6$ to get non-extendability of general $K3$ surfaces of index $1$ and genus $6q-7$ or $a = 4, e = 0, b = q$ with $2 \nmid q$ with $q \geq 4$ to get non-extendability of general $K3$ surfaces of index $1$ and genus $8q+1$ or $a = p, e = 0, b = q$ with $p, q$  with $p \geq 5, q \geq 3$, $p, q$ coprime to get non-extendability of general $K3$ surfaces of index $1$ and genus $2pq+1$ or $a = p, e = 1, b = q$ with $p, q$ coprime, $p \geq 4, q \geq p+2$ to get non-extendability of general $K3$ surfaces of index $1$ and genus $2pq-p^2+1$. 
     
    \end{enumerate}
\end{proof}

\begin{corollary}\label{corank one}
\textcolor{black}{Let $X \subset \mathbb{P}^{M}$ be a general $K3$ surface in $\mathcal{H}_{r,g}$ embedded by the complete linear series of $rB$, where $B$ is a primitive very ample line bundle, with $B^2 = 2g-2$ and $M = 1+r^2(g-1)$.} Let $C \in |rB|$ be any smooth curve section of $X$. Assume either $r \geq 5, g \geq 3$ or $r = 4, g \geq 4$ or $r = 3, g \geq 5$ or $r = 2$, $g \geq 7$ or $r = 1$ and $g$ is either $g = 4k+1, k \geq 5$ or $g = 18k+4, k \geq 1$, or $g = 18k+7, k \geq 2$ or $g = 18k+13, k \geq 1$ or $g = 18k+16, k \geq 1$ Then $\alpha(C) = \operatorname{cork}(\Phi_{\omega_C}) = 1$.
\end{corollary}

\begin{proof}
First note that $2g(C)-2 = r^2B^2 = r^2(2g-2)$ and hence $g(C) = r^2(g-1)+1 \geq 11$. Now let $\operatorname{cork}(\Phi_{\omega_C}) = r$. Then by \cite[$2.2$, $2.2.2$]{CDS20} (see also \cite{ABS17}), we have the existence of an arithmetically Gorenstein normal variety $W$, such that every ribbon $[v] \in \mathbb{P}(\operatorname{ker}(^T\Phi_{\omega_C}))$ (these are ribbons on $C$ with conormal bundle $\omega_C^{-1}$) is contained in a unique section of $W$ by a linear $g(C)$ space. In particular, $W$ contains all surface extensions of $C$. Therefore $\alpha(X) = 0$ forces $\alpha(C) = 1$. 
\end{proof}

\begin{remark}{\label{independent corank one}}
The corank one theorem can also be derived from our methods independently by showing that the map $H^1(N_{\widetilde{Y}/\mathbb{P}^M}(-2\widetilde{H})) \to H^1(N_{\widetilde{Y}/\mathbb{P}^M}(-\widetilde{H}))$ in the exact sequence
$$0 \to H^0(N_{\widetilde{Y}/\mathbb{P}^M}(-2\widetilde{H})) \to H^0(N_{\widetilde{Y}/\mathbb{P}^M}(-\widetilde{H})) \to H^0(N_{\widetilde{C}/\mathbb{P}^M}(-\widetilde{H})) \to H^1(N_{\widetilde{Y}/\mathbb{P}^M}(-2\widetilde{H})) \to H^1(N_{\widetilde{Y}/\mathbb{P}^M}(-\widetilde{H}))$$
is injective. To show this, along with the cohomological computations and analysis of coboundary maps using Equations \ref{1}, \ref{2}, \ref{3}, as introduced in here, one needs to use results on Gaussian maps on Hirzebruch surfaces in \cite{DM92} and \cite{DP12}. Since showing injectivity of this map is essential to handle the case of extendability of prime $K3$ surfaces, we will come back to this in a forthcoming article (\cite{BM25}). 
\end{remark}

\Cref{calculating beta} has some consequences on extendability of $K3$ double covers $X$ of $\mathbb{F}_0$ or $\mathbb{F}_1$. First one notices by $T_0 = \pi^*C_0$ and $T_1 = \pi^*f$ generate the Picard group of $X$. 

\begin{theorem}\label{rank two}

Let $X$ be a $K3$ surface which is a double cover $\pi: X \to Y$ of $Y = \mathbb{F}_0$ or $Y = \mathbb{F}_1$. Let $T_0 = \pi^*C_0$ and $T_1 = \pi^*f$. Then for $X \subset \mathbb{P}^M$ where the embedding is induced by the complete linear series of a very ample line bundle $B = aT_0+bT_1$, we have $\alpha(X) = 0$ if $a = 3, b \geq 5, e = 0$ or a = $3, b \geq 7, e = 1$ respectively.

\end{theorem}

\begin{proof}
    First note that by \cite[Proposition $2.1$]{DM24}, for the split $K3$ carpet $\widetilde{Y}$ on $Y = \mathbb{F}_0$ or $Y = \mathbb{F}_1$, the line bundle $\widetilde{H}$ such that $\widetilde{H}|_Y = H$ is very ample and in fact projectively normal if $a \geq 2$ and $b \geq (a-1)e+2$. Further, since any double cover $X$ isotrivially degenerates to the split ribbon $\widetilde{Y}$, $\alpha(X) = 0$ for $X$, if $\alpha(\widetilde{Y}) = 0$. Now the vanishing of the latter follows from part \Cref{calculating beta}, $(1)$.
\end{proof}

\subsection{Extendability of general non-prime $K3$ surfaces of lower genus}

In the following theorem, we compute $\alpha(X)$ for a general $K3$ surface $X \subset \mathbb{P}^M$ of index $r \geq 2$ and low genus $g$. We will use the upper bounds to compute the dimension of the tangent space corresponding to the Hilbert point represented by the cone over a $K3$ surface. We show later in \Cref{equality of alpha(x)} that the upper bounds are actually equalities. 

\begin{theorem}\label{ZL for lower genus $K3$ surfaces}
\textcolor{black}{Let $X \subset \mathbb{P}^{M}$ be a general $K3$ surface in $\mathcal{H}_{r,g}$ embedded by the complete linear series of $rB$, where $B$ is a primitive very ample line bundle, with $B^2 = 2g-2$ and $M = 1+r^2(g-1)$.}
 
    \begin{enumerate}
       % \item If $r = 1$, $g = 12$, then $\alpha(X) \leq 1$
       % \item If $r = 1$, $g = 10$, then $\alpha(X) \leq 3$  
       % \item If $r = 1$, $g = 9$, then $\alpha(X) \leq 4$
       % \item If $r = 1$, $g = 8$, then $\alpha(X) \leq 6$
       % \item If $r = 1$, $g = 7$, then $\alpha(X) \leq 8$
       % \item If $r = 1$, $g = 6$, then $\alpha(X) \leq 10$ (\textcolor{blue}{exceptional case})
       % \item \label{2,2} If $r = 2$, $g = 2$, then $\alpha(X) \leq 14$ 
        \item \label{2,3} If $r = 2$, $g = 3$, then $\alpha(X) \leq 10$ %(\textcolor{blue}{exceptional case})
        %(\textcolor{blue}{Needs improvement by $1$.})
        \item \label{2,4} If $r = 2$, $g = 4$, then $\alpha(X) \leq 6$ 
        \item \label{2,5} If $r = 2$, $g = 5$, then $\alpha(X) \leq 3$ 
        \item \label{2,6} If $r = 2$, $g = 6$, then $\alpha(X) \leq 1$ 
      %  \item If $r = 3$, $g = 2$, then $\alpha(X) \leq 10$ %(\textcolor{blue}{exceptional case})
        \item \label{3,3} If $r = 3$, $g = 3$, then $\alpha(X) \leq 4$
        \item \label{3,4} If $r = 3$, $g = 4$, then $\alpha(X) \leq 1$ 
        \item \label{4,3} If $r = 4$, $g = 3$, then $\alpha(X) \leq 1$

    \end{enumerate}
    
\end{theorem}

\begin{proof}

We first prove for $r = 2$. Since $3 \leq g \leq 6$. We can degenerate such a $K3$ surface to a $K3$ carpet $\widetilde{Y} \subset \mathbb{P}^M$, extending $Y \subset \mathbb{P}^M$, where $Y = \mathbb{F}_e$ embedded by the complete linear series of $2C_0+bf$, where $b = 2m$ and $e = 0$ or $e = 1$ depending on whether $g$ is odd or even respectively. In this case $g = 2m-e+1$. So it is enough to find an upper bound for $\alpha(\widetilde{Y})$ for a $K3$ carpet $\widetilde{Y} \subset \mathbb{P}^M$, extending $Y \subset \mathbb{P}^M$, where $Y = \mathbb{F}_e$ embedded by the complete linear series of $2C_0+2mf$ where the pair $(e,m)$ is one of $(0,1), (0,2), (1,2), (1,3)$. We have shown in Theorem \ref{ZL for K3 carpets}, part $(2)$, that $$\alpha(\widetilde{Y})  \leq  \gamma \leq h^1(T_Y(-H+K_Y))+h^0(N_{Y/\mathbb{P}^M}(-H)+h^0(-H-2K_Y)-M-1 $$
So we calculate the right hand side of the above equation for the above mentioned values of $(e,m)$. By Lemma \ref{T_Y(-H)}, $(2)$, we have that for $m \geq e+1/2$, $h^1(T_Y(-H+K_Y)) = 0$. Now in the proof of part $(2)$ of Theorem \ref{calculating beta}, using equations \ref{35}-\ref{41}, we have shown that if $m \geq e/2+3/2$, $h^0(N_{Y/\mathbb{P}^M}(-H) = M+1$. Hence except for the case for the cases $(e,m) = (0,1)$, $\alpha(\widetilde{Y}) \leq h^0(-H-2K_Y)$. This gives us the upper bounds as stated in the statement for the cases $r = 2, 4 \leq g \leq 6$. On the other hand, in the case $(0,1)$, once again in the proof of part $(2)$ of Theorem \ref{calculating beta}, in equation \ref{38}, we have that $H^1(p^*(T_{\mathbb{P}^1})(-H)) = H^1(\mathcal{O}_{\mathbb{P}^1}(-2) = \mathbb{C}$ for $(0,1)$ 
%and $H^1(p^*(T_{\mathbb{P}^1})(-H)) = H^1(\mathcal{O}_{\mathbb{P}^1}(-3) = \mathbb{C}^2$ for $(1,1)$ 
which in turn gives us $h^0(N_{Y/\mathbb{P}^M}(-H)) \leq M+2$ 
%and $h^0(N_{Y/\mathbb{P}^M}(-H) \leq M+3$ respectively 
from the sequence in \ref{35}. Hence in this case $\alpha(\widetilde{Y}) \leq h^0(-H-2K_Y)+1$.
%respectively $\alpha(\widetilde{Y}) \leq h^0(-H-2K_Y)+2$. 
This gives us the remaining case $r = 2, g = 3$ 
%and $r = 2, g = 2$. \par

We now prove for $r = 3,4$. We can degenerate such a $K3$ surface to a $K3$ carpet $\widetilde{Y} \subset \mathbb{P}^M$, extending $Y \subset \mathbb{P}^M$, where $Y = \mathbb{F}_e$ embedded by the complete linear series of $rC_0+rmf$, where $e = 0$ or $e = 1$ depending on whether $g$ is odd or even respectively. Again we have $g = 2m-e+1$. So it is enough to find an upper bound for $\alpha(\widetilde{Y})$ for a $K3$ carpet $\widetilde{Y} \subset \mathbb{P}^M$, extending $Y \subset \mathbb{P}^M$, where $Y = \mathbb{F}_e$ embedded by the complete linear series of $rC_0+rmf$ where the tuple $(r,e,m)$ is one of $(3,0,1), (3,1,2), (4,0,1)$. In each case we calculate $\beta$ in Theorem \ref{ZL for K3 carpets} which gives an upper bound for $\alpha(\widetilde{Y})$. In each case 
%other than $(3,1,1)$ 
one can check that the only contribution to $\beta$ comes from $h^0(-H-2K_Y)$, which gives us the stated bounds. 
%In the case $(3,1,1)$, the contribution comes from $h^0(-H-2K_Y)+h^1(T_Y(-H))$. 
\end{proof}

\section{Classification of non-prime Fano-threefolds and Mukai varieties of Picard rank one and their Hilbert schemes}

Boundedness and classification of smooth Fano threefolds and more generally varieties with a canonical curve section (Mukai varieties) with Picard rank one was first obtained by Iskovskih-Mori-Mukai in \cite{I77}, \cite{I78}, \cite{M88} and \cite{MM82} (see also \cite{BKM25}). We give a new proof in the non-prime case using completely different methods and a general principle outlined in \cite{CLM98}. \par
In the following theorem we compute $h^0(N_{\widetilde{Y}/\mathbb{P}^M}(-2\widetilde{H}))$ which will be an upper bound for $h^0(N_{X/\mathbb{P}^M}(-2))$ for a general non-prime $K3$ surface. We will need this to compute the tangent space to the Hilbert scheme of non-prime Fano threefolds at the Hilbert point representing the cone over such a $K3$ surface. Further, the vanishing of this cohomology group allows us to recover Wahl's theorem on the non-surjectivity of the Gaussian-Wahl map $\Phi_{\omega_C}$ for a curve lying on a $K3$ surface in the non-prime case.

\begin{proposition}\label{H^0(N(-2))}
Let $\widetilde{Y} \hookrightarrow \mathbb{P}^M$ be a $K3$ carpet embedded by the complete linear series of a very ample line bundle $\widetilde{H}$ such that the possibly degenerate embedding $Y = \mathbb{F}_e \hookrightarrow \mathbb{P}^M$ induced on the reduced part is given by the complete linear series of a very ample line bundle $H = aC_0+bf$ followed by a linear embedding of projective spaces.   
Assume $H^1(-H+K_Y) = 0$. Then $$h^0(N_{\widetilde{Y}/\mathbb{P}^N}(-2\widetilde{H})) \leq h^1(T_Y(-2H+K_Y)) + h^1(T_Y(-2H)) - h^0(T_Y(-2H)) +  h^1(-2H-K_Y) + h^0(-2H-2K_Y) $$
Consequently, $h^0(N_{\widetilde{Y}/\mathbb{P}^N}(-2\widetilde{H})) = 0$ if either of the following conditions hold
\begin{enumerate}
    \item $a = 2, b \geq 3, e = 0$
    \item $a \geq 3, b \geq 2, e = 0$
    \item $a \geq 2, b \geq a/2+1, e = 1$
    \item $a \geq 2, b \geq ae/2+1/2, e \geq 2$
    
\end{enumerate}
Moreover, if $a = 2, b = 2, e = 0$, $h^0(N_{\widetilde{Y}/\mathbb{P}^N}(-2\widetilde{H})) \leq 1$. In this case, $h^0(N_{\widetilde{Y}/\mathbb{P}^N}(-3\widetilde{H})) = 0$. Hence if $(r = 2, g \geq 4), (r \geq 3, g \geq 3)$, $H^0(N_{X/\mathbb{P}^M}(-2)) = 0$. Further, if $(r = 2, g = 3)$, then $h^0(N_{X/\mathbb{P}^M}(-2)) \leq 1$.
\end{proposition}

\begin{proof}
    %We follow the proof of \cite[Theorem $2.2$]{BM24}, and indicate the changes.  
   Let $k \geq 2$. We tensor \Cref{1} by $-k\widetilde{H}$ to obtain
    $$h^0(N_{\widetilde{Y}/\mathbb{P}^M}(-k\widetilde{H})) \leq  h^0(N_{\widetilde{Y}/\mathbb{P}^M}(-kH+K_Y)) + h^0(N_{\widetilde{Y}/\mathbb{P}^M}(-k\widetilde{H}))$$
    Now tensoring \Cref{2} by $(-kH+K_Y)$ and noting that $h^0(-kH-K_Y) = 0$ we have,
    $$h^0(N_{\widetilde{Y}/\mathbb{P}^M}(-kH+K_Y)) = h^0(\mathcal{H}om(I_{\widetilde{Y}}/I_{Y}^2, \mathcal{O}_Y)(-kH+K_Y))$$
    Once again, tensoring \Cref{3}, by $-kH+K_Y$ and noting that $h^0(-kH) = h^1(-kH) = 0$, we have 
    $$h^0(\mathcal{H}om(I_{\widetilde{Y}}/I_{Y}^2, \mathcal{O}_Y)(-kH+K_Y)) = h^0(N_{Y/\mathbb{P}^M}(-kH+K_Y))$$

    Now tensoring \Cref{2} by $(-kH)$ we have,
    $$h^0(N_{\widetilde{Y}/\mathbb{P}^M}(-kH)) \leq h^0(\mathcal{H}om(I_{\widetilde{Y}}/I_{Y}^2, \mathcal{O}_Y)(-kH))+h^0(-kH-2K_Y)$$
    Once again, tensoring \Cref{3}, by $-kH$ and noting that $h^0(-kH-K_Y)) = 0$, we have 
    $$h^0(\mathcal{H}om(I_{\widetilde{Y}}/I_{Y}^2, \mathcal{O}_Y)(-kH+K_Y)) = h^0(N_{Y/\mathbb{P}^M}(-kH))+h^1(-kH-K_Y)$$
    
    %instead of $-\widetilde{H}$ and follow the proof verbatim. Equations $2.1-2.9$ in \cite[Theorem $2.2$]{BM24}, yields 
    Hence we have,
    $$h^0(N_{\widetilde{Y}/\mathbb{P}^M}(-k\widetilde{H})) \leq  h^0(N_{Y/\mathbb{P}^M}(-kH+K_Y)) + h^0(N_{Y/\mathbb{P}^M}(-kH)) +  h^1(-kH-K_Y) + h^0(-kH-2K_Y)$$

Now \Cref{9} and \Cref{11} yield %tensoring the exact sequence $2.10-2.11$ in \cite[Theorem $2.2$]{BM24} by $-kH$ now yields 
$$h^0(N_{Y/\mathbb{P}^M}(-kH)) \leq h^1(T_Y(-kH)) - h^0(T_Y(-kH))$$
%Similarly tensoring $2.10-2.11$ in \cite[Theorem $2.2$]{BM24} by $-kH+K_Y$ yields 
and
$$h^0(N_{Y/\mathbb{P}^M}(-kH+K_Y)) \leq h^1(T_Y(-kH+K_Y))$$
Hence we have 
$$h^0(N_{\widetilde{Y}/\mathbb{P}^N}(-k\widetilde{H})) \leq h^1(T_Y(-kH+K_Y)) + h^1(T_Y(-kH)) - h^0(T_Y(-kH)) +  h^1(-kH-K_Y) + h^0(-kH-2K_Y) $$
Now the rest of the proof follows from \Cref{-H-K_Y}, \Cref{-H-2K_Y}, \Cref{T_Y(-H)}. For the ease of the reader we point out that in the case $a = 2, b = 2, e = 0$, the only term that contributes in the upper bound of $h^0(N_{\widetilde{Y}/\mathbb{P}^N}(-2\widetilde{H}))$ is $h^0(-2H-2K_Y) = 1$. 
\end{proof}

%\begin{proposition}\label{r = 1, H^0(N(-2))}
%Let $\widetilde{Y} \hookrightarrow \mathbb{P}^M$ be a $K3$ carpet embedded by the complete linear series of a very ample line bundle $\widetilde{H}$ such that the embedding induced on the reduced part is $Y = \mathbb{F}_e \hookrightarrow \mathbb{P}^N$ embedded by the complete linear series of a very ample line bundle $H = aC_0+bf$. Assume $H^1(-H+K_Y) = 0$. Then $$h^0(N_{\widetilde{Y}/\mathbb{P}^N}(-2\widetilde{H})) \leq h^1(T_Y(-2H+K_Y)) + h^0(N_{Y/\mathbb{P}^M}(-2H))+ h^0(-2H-2K_Y) $$
%Consequently, $h^0(N_{\widetilde{Y}/\mathbb{P}^N}(-2\widetilde{H})) = 0$ if $a = 1, b \geq e+5/2$.
%Moreover, if $a = 1, b = 3, e = 1$, $h^0(N_{\widetilde{Y}/\mathbb{P}^N}(-2\widetilde{H})) \leq 1$
%\end{proposition}

\begin{definition}\label{Hilbert scheme of Fanos and K3}
    Let \textcolor{black}{$\mathcal{V}_{r,g}$ denote the Hilbert scheme of smooth anticanonically embedded Fano threefolds of index $r$ and genus $g$, i.e, anticanonically embedded Fano threefolds $V \subset \mathbb{P}^{2+r^2(g-1)}$ such that $-K_V$ is $r-$ divisible in $\operatorname{Pic}(V)$. Let $\mathcal{V}_{r,g,1}$ denote the closure inside $\mathcal{V}_{r,g}$ of the locus of points representing the Fano threefolds with Picard rank one.} Recall that $\mathcal{H}_{r,g}$ denotes the Hilbert scheme of $K3$ surfaces of index $r$ and genus $g$, i.e, $K3$ surfaces $X \subset \mathbb{P}^{1+r^2(g-1)}$ embedded by the complete linear series of $rB$ where $B$ is a primitive very ample line bundle with $B^2 = 2g-2$. \textcolor{black}{Denote by $C(X) \subset \mathbb{P}^{2+r^2(g-1)}$ the cone over such a $K3$ surface.} Let $\mathcal{F}_{r,g}$ denote the flag Hilbert scheme consisting of pairs $(V,X)$ where $V \in \mathcal{V}_{r,g,1}$ and $X$ is a hyperplane section of $V$ and hence $X \in \mathcal{H}_{r,g}$. There is natural projection map from $p: \mathcal{F}_{r,g} \to \mathcal{H}_{r,g}$.
\end{definition}

\color{black}
 We proved \Cref{corank one} by combining our results in \Cref{non-extendability of K3 surfaces} with \cite{CDS20} and \cite{ABS17}. However, using \Cref{H^0(N(-2))}, we independently recover Wahl's theorem on non surjectivity of the Gaussian-Wahl map (see \cite{W87},  \cite{BM87}, \cite{CM90}, \cite{V92}) in the non-prime case.

\begin{corollary}\label{Wahl}
 Let $X$ be a $K3$ surface and $C \subset X$ is a smooth curve with $C \in |rB|$ with $r \geq 2$. Then the Wahl map $\Phi_{\omega_C}$ is not surjective. 
\end{corollary}

\begin{proof}
   Let $B^2 = 2g-2$. The case $g = 2$, can be worked out separately (for example, see \cite{CLM98}). For $r = 2, g = 3$, $g(C) = 9$ and $\Phi_{\omega_C}$ cannot be surjective since the domain is of dimension less than the target. For $r = 2, g \geq 4$ or $r \geq 3, g \geq 3$, consider the Hilbert scheme $\mathcal{H}_{r,g}$. It is enough to prove the statement for a general element of $\mathcal{H}_{r,g}$ since surjectivity of $\Phi_{\omega_C}$ is an open condition. Let $M = 1+r^2(g-1)$ Now we have an exact sequence
   $$0 \to H^0(N_{X/\mathbb{P}^M}(-2)) \to H^0(N_{X/\mathbb{P}^M}(-1)) \to H^0(N_{C/\mathbb{P}^{M-1}}(-1))$$
   By \Cref{H^0(N(-2))}, $H^0(N_{X/\mathbb{P}^M}(-2)) = 0$ and hence we have that $\operatorname{corank}(\Phi_{\omega_C}) = h^0(N_{C/\mathbb{P}^{M-1}}(-1))-M \geq h^0(N_{X/\mathbb{P}^M}(-1))-M \geq 1$.
\end{proof}

\color{black}

\begin{corollary}\label{upper bound on cone over $K3$ surface}

\textcolor{black}{Let $X \subset \mathbb{P}^{M}$ be a general $K3$ surface in $\mathcal{H}_{r,g}$ embedded by the complete linear series of $rB$, where $B$ is a primitive very ample line bundle, with $B^2 = 2g-2$ and $M = 1+r^2(g-1)$.}

\begin{enumerate}
    \item If $r \geq 2, g \geq 4$. Then the dimension of the tangent space of $\mathcal{V}_{r,g,1}$ at the point $C(X) \subset \mathbb{P}^{M+1}$ is given by 

\begin{equation}\label{upper bounds}
    \operatorname{dim}T_{C(X)}(\mathcal{V}_{r,g,1}) = h^0(N_{X/\mathbb{P}^M})+h^0(N_{X/\mathbb{P}^{M}}(-1)) = 18+(2+r^2(g-1))^2+\alpha(X)+(2+r^2(g-1))
\end{equation}

   \item If $r = 2, g = 3$. Then the dimension of the tangent space of $\mathcal{V}_{r,g,1}$ at the point $C(X) \subset \mathbb{P}^{M+1}$ is given by 

\begin{equation}\label{upper bounds for r = 2, g = 3}
    \operatorname{dim}T_{C(X)}(\mathcal{V}_{r,g,1}) \leq h^0(N_{X/\mathbb{P}^M})+h^0(N_{X/\mathbb{P}^{M}}(-1)) + 1 = 18+(2+r^2(g-1))^2+\alpha(X)+(2+r^2(g-1))+1
\end{equation}

\end{enumerate}

\end{corollary}

\begin{proof}
    The proof of the corollary follows from \Cref{H^0(N(-2))} and the fact that $$\operatorname{dim}T_{C(X)}(\mathcal{V}_{r,g,1}) = \Sigma_{k \geq 0} h^0(N_{X/\mathbb{P}^M}(-k))$$
\end{proof}

\begin{lemma}\label{hyperplane section of a general fano is a general K3}

Let $V$ be a general element of $\mathcal{V}_{r,g,1}$. Then a general hyperplane section of $V$ corresponds to a general $K3$ surface in $\mathcal{H}_{r,g}$.
    
\end{lemma}

\begin{proof}
    The proof follows from standard deformation theory arguments (see for example \cite[Lemma $3.7$]{CLM98}, or \cite{Bea02}).
\end{proof}

\begin{theorem}\label{classification of fanos with index two}
\textcolor{black}{Consider} $\mathcal{V}_{r,g,1}$ \textcolor{black}{for} $r \geq 2$ and genus $g \geq 3$. Then 
\begin{enumerate}
    \item $\mathcal{V}_{r,g,1} = \varnothing$ if $(r = 2, g \geq 7)$, $(r = 3, g = 3)$, $(r = 3, g \geq 5)$, $(r = 4, g \geq 4)$ or $(r = 5, g \geq 3)$. 
    \item For $(r = 2, 3 \leq g \leq 6), (r = 3, g = 4), (r = 4, g = 3)$, $\mathcal{V}_{r,g,1}$ is irreducible, and the families of Fano and Iskovskih form an open dense irreducible subset of $\mathcal{V}_{r,g,1}$. The general fiber of the projection map $p: \mathcal{F}_{r,g} \to \mathcal{H}_{r,g}$ as defined in \Cref{Hilbert scheme of Fanos and K3} is irreducible. Upto projective transformations, a general $K3$ surface in $\mathcal{H}_{2,6}, \mathcal{H}_{3,4}$ or $ \mathcal{H}_{4,3}$ is contained in a unique Fano threefold in $\mathcal{V}_{2,6,1}, \mathcal{V}_{3,4,1}$ or $ \mathcal{V}_{4,3,1}$ respectively.
\end{enumerate}
\end{theorem}

\begin{proof}
  \noindent\textit{Proof of $(1)$} First note that for every $r$, there is a unique component $\mathcal{H}_{r,g}$ of the Hilbert scheme parameterizing $K3$ surfaces $X \subset \mathbb{P}^{h^0(rB)-1}$ embedded by the complete linear of a very ample line bundle $rB$, where $B^2 = 2g-2$. Now excepting the case, $(r = 3, g = 3)$, part $(1)$ follows from \Cref{non-extendability of K3 surfaces}. For the case $(r=3, g=3)$, note that if the hyperplane section of a Fano threefold $V$ in  $\mathcal{V}_{r,g,1}$ be denoted by $rL$ and $L^3 = d$, then computing $h^0(rL) = h^0(rB)-1$ by Riemann-Roch theorem, we have $d = 2(g-1)/r$. But for $r = 3, g = 3$, $d$ cannot be an integer. This resolves the case. \par
  \textit{Proof of $(2)$} First one notes that for each $r$ and $g$, $\mathcal{H}_{r,g}$ is irreducible. Now the list of examples of Fano-Iskovskih-Mori-Mukai (see for example \cite[Table $3.1$]{CLM98}), show that in each of the cases mentioned by part $(2)$, there exist $K3$ surfaces in $\mathcal{H}_{r,g}$ which are extendable to smooth Fano threefolds. Hence by \Cref{hyperplane section of a general fano is a general K3}, in each of the cases, $\mathcal{H}_{r,g}$ is dominated by $\mathcal{V}_{r,g,1}$. We end the proof by showing, $\mathcal{V}_{r,g,1}$ has a unique irreducible component and that the the examples of Fano-Iskovskih-Mori-Mukai form an open dense subset of $\mathcal{V}_{r,g,1}$.  To see this, we observe first as in \cite{CLM93} and \cite{CLM98}, that a projectively Cohen-Macaulay scheme degenerates along a flat family to the cone over its hyperplane section. Hence once again using, \Cref{hyperplane section of a general fano is a general K3}, it is enough to show that for a general $K3$ surface $X \subset \mathbb{P}^{M}$ in $\mathcal{H}_{r,g}$, the cone $C(X) \subset \mathbb{P}^{M+1}$ is a smooth point of $\mathcal{V}_{r,g,1}$. For the stated values of pairs $(r,g)$, 
 % \Cref{upper bounds} and \Cref{upper bounds for r = 2, g = 3} in 
  \Cref{upper bound on cone over $K3$ surface} along with \Cref{ZL for lower genus $K3$ surfaces}, gives an upper bound to $\operatorname{dim}T_{C(X)}(\mathcal{V}_{r,g,1})$ which is listed as follows 
  \begin{enumerate}
     % \item $\operatorname{dim}T_{C(X)}(\mathcal{V}_{1,6,1}) \leq 85$

     % \item $\operatorname{dim}T_{C(X)}(\mathcal{V}_{1,7,1}) \leq 98$

     % \item $\operatorname{dim}T_{C(X)}(\mathcal{V}_{1,8,1}) \leq 114$

     % \item $\operatorname{dim}T_{C(X)}(\mathcal{V}_{1,9,1}) \leq 132$
      
     % \item $\operatorname{dim}T_{C(X)}(\mathcal{V}_{1,10,1}) \leq 153$

     % \item $\operatorname{dim}T_{C(X)}(\mathcal{V}_{1,12,1}) \leq 201$
      
      \item $\operatorname{dim}T_{C(X)}(\mathcal{V}_{2,3,1}) \leq 139$

      \item $\operatorname{dim}T_{C(X)}(\mathcal{V}_{2,4,1}) \leq 234$

      \item $\operatorname{dim}T_{C(X)}(\mathcal{V}_{2,5,1}) \leq 363$

      \item $\operatorname{dim}T_{C(X)}(\mathcal{V}_{2,6,1}) \leq 525$

      \item $\operatorname{dim}T_{C(X)}(\mathcal{V}_{3,4,1}) \leq 889$

      \item $\operatorname{dim}T_{C(X)}(\mathcal{V}_{4,3,1}) \leq 1209$
  \end{enumerate}
Now for a smooth Fano threefold $V \in \mathcal{V}_{r,g,1}$, $V$ represents a smooth point of $\mathcal{V}_{r,g,1}$, since embedded deformations of an anticanonically embedded Fano variety of dimension at least $3$ can be shown to be unobstructed using the vanishing of $h^1(N_{V/\mathbb{P}^{M+1}})$ by Kodaira-Nakano vanishing theorem. Therefore the dimension of $\mathcal{V}_{r,g,1}$ at a smooth point is given by $\operatorname{dim}T_V(\mathcal{V}_{r,g,1}) = h^0(N_{V/\mathbb{P}^{M+1}})$. Now using the descriptions of the examples of Fano-Iskovskih-Mori-Mukai, one can calculate the dimension $h^0(N_{V/\mathbb{P}^{M+1}})$ (see for example \cite[Table $2$, pg $666$]{CLM93} and \cite[Table $3.1$]{CLM98}) to check that in each case, $\operatorname{dim}T_V(\mathcal{V}_{r,g,1})$ is exactly the upper bound just obtained. Hence $C(X)$ is also a smooth point of $\mathcal{V}_{r,g,1}$. \par
To see the statement on the irreducibility of fibers of $p$, note that for a general $X \subset V$, we have an exact sequence 
$$0 \to H^0(N_V(-1)) \to H^0(N_V) \to H^0(N_X) \to 0$$
The surjectivity arises from the fact that from \Cref{hyperplane section of a general fano is a general K3}, one can show that every embedded deformation of $X$ can be lifted to an embedded deformation of $V$. Since $H^1(N_{V/\mathbb{P}^{M+1}}) = 0$, this implies $H^1(N_{V/\mathbb{P}^{M+1}}(-1)) = 0$. This implies that $V$ is a smooth point of the fiber and that the dimension of the fiber at $V$ is $h^0(N_V(-1))$. Now once again the cone $C(X)$ lies in every irreducible component of the fiber and the tangent space to the fiber at $C(X)$ is given by 
$$h^0(N_X(-1))+h^0(N_X(-2))$$
But we already checked that 
$$h^0(N_V(-1)) = h^0(N_V) - h^0(N_X) = h^0(N_X(-1))+h^0(N_X(-2))$$
This implies that $C(X)$ is a smooth point of the fiber and hence the fiber is irreducible. \par
The last statement follows from the fact that $\alpha(X) = 1$ in these cases.

\end{proof}

\color{black}
\begin{definition}\label{Mukai varieties}
Let $\mathcal{V}_{n,r,g}$ denote the Hilbert scheme of smooth embedded Fano varieties of dimension $n \geq 4$, index $r(n-2)$ and genus $g$ with a canonical curve section. The surface sections of such varieties are $K3$ surfaces $X \in \mathcal{H}_{r,g}$ and hence their curve sections are canonical curves of genus $1+r^2(g-1)$. Denote by $\mathcal{V}_{n,r,g,1}$, the closure in $\mathcal{V}_{n,r,g}$ of the locus parameterizing the varieties with Picard number one. Such varieties are called Mukai varieties.
\end{definition}

\color{black}

\begin{theorem}\label{classification of Mukai varieties with index two}
\textcolor{black}{Consider} $\mathcal{V}_{n,r,g,1}$ \textcolor{black}{with index $r \geq 2$ and genus $g \geq 3$.} Then 
\begin{enumerate}
    \item $\mathcal{V}_{n,r,g,1} = \varnothing$ if $(r = 2, 3 \leq g \leq 4)$, $(n \geq 6, r = 2, g = 5)$, $(r = 2, g \geq 6)$, $(r = 3, g \geq 3)$, or $(r = 4, g \geq 3)$ . 
    \item $\mathcal{V}_{4,2,5,1}$ and $\mathcal{V}_{5,2,5,1}$ are irreducible, and the families of Mukai form an open dense irreducible subset of the components.
\end{enumerate}
\end{theorem}

\begin{proof}
   By \Cref{non-extendability of K3 surfaces}, the cases to deal with are the lower genus cases $(1)-(7)$  listed in \Cref{upper bound on cone over $K3$ surface}. Now $\mathcal{V}_{n,r,g,1} = \varnothing$ for $n \geq 4$ for cases \ref{2,6}, \ref{3,4}, \ref{4,3} since $\alpha(X) \leq 1$. $\mathcal{V}_{n,r,g,1} = \varnothing$ for $n \geq 3$ for case \ref{3,3} has been shown in \Cref{classification of fanos with index two} part $(1)$. Similarly $\mathcal{V}_{n,r,g,1} = \varnothing$ for $n \geq 6$ for case \ref{2,5} since $\alpha(X) \leq 3$.  $\mathcal{V}_{n,r,g,1} = \varnothing$ for $n \geq 4$, for cases \ref{2,3}, \ref{2,4}, can be shown by basic adjunction theory arguments, for example see \cite[Theorem $3.11$]{CLM98}. It remains to show part \ref{2,5}. It is enough to show the statement for $\mathcal{V}_{5,2,5,1}$. In this case by the same arguments in \Cref{upper bound on cone over $K3$ surface}, we see that dimension of the tangent space $T_{C(X)}(\mathcal{V}_{5,2,5,1})$ to the (triple) cone $C(X) \in \mathcal{V}_{5,2,5,1}$ over $X \in \mathcal{V}_{2,5,1}$ is given by 
   $$\operatorname{dim}T_{C(X)}(\mathcal{V}_{5,2,5,1}) = 18+2(r^2(g-1))^2+3(\alpha(X)+2+r^2(g-1)) = 405$$
   Now the result follows by comparing with the number of parameters of Mukai's family of example.

\end{proof}

\begin{corollary}\label{equality of alpha(x)}
Let \textcolor{black}{$X \subset \mathbb{P}^{M}$ be a general $K3$ surface in $\mathcal{H}_{r,g}$ embedded by the complete linear series of $rB$, where $B$ is a primitive very ample line bundle, with $B^2 = 2g-2$ and $M = 1+r^2(g-1)$.}

    \begin{enumerate}
       % \item If $r = 1$, $g = 12$, then $\alpha(X) \leq 1$
       % \item If $r = 1$, $g = 10$, then $\alpha(X) \leq 3$  
       % \item If $r = 1$, $g = 9$, then $\alpha(X) \leq 4$
       % \item If $r = 1$, $g = 8$, then $\alpha(X) \leq 6$
       % \item If $r = 1$, $g = 7$, then $\alpha(X) \leq 8$
       % \item If $r = 1$, $g = 6$, then $\alpha(X) \leq 10$ (\textcolor{blue}{exceptional case})
        \item If $r = 2$, $g = 3$, then $\alpha(X) = 10$ %(\textcolor{blue}{exceptional %case})
        %(\textcolor{blue}{Needs improvement by $1$.})
        \item If $r = 2$, $g = 4$, then $\alpha(X) = 6$ 
        \item If $r = 2$, $g = 5$, then $\alpha(X) = 3$ 
        \item If $r = 2$, $g = 6$, then $\alpha(X) = 1$ 
        \item If $r = 3$, $g = 3$, then $\alpha(X) \leq 4$
        \item If $r = 3$, $g = 4$, then $\alpha(X) = 1$ 
        \item If $r = 4$, $g = 3$, then $\alpha(X) = 1$
\end{enumerate}
\end{corollary}

\begin{proof}
    Excepting in case $(5)$, we have by \Cref{classification of fanos with index two}, $\mathcal{V}_{r,g,1}$ is non empty and that $T_{C(X)}(\mathcal{V}_{r,g,1}) = T_V(\mathcal{V}_{r,g,1})$ where $V$ is a smooth point of $\mathcal{V}_{r,g,1}$. Hence by \Cref{upper bound on cone over $K3$ surface}, we conclude that the upper bounds to $\alpha(X)$ in \Cref{ZL for lower genus $K3$ surfaces}, are actually equalities. 
\end{proof}

\section{Degeneration of $K3$ surfaces to union of Hirzebruch surfaces embedded by arbitrary linear series}

\label{main results}

In this section we show how $K3$ carpets arise as degenerations of snc union of embedded Hirzebruch surfaces intersecting along an anticanonical elliptic curve. This relates our degeneration to that of \cite{CLM93}. In \cite{CLM93}, the authors degenerate a prime $K3$ surface into a union of scrolls meeting along an elliptic curve and further degenerate such a union into a union of planes whose hyperplane sections are graph curves with corank one Gaussian maps. Similarly, one can think of our degeneration as a degeneration into a union of Hirzebruch surfaces embedded as non-scrolls, but instead of further degenerating to a union of planes, we degenerate into $K3$ carpets. To fix some notations, we start with a remark.

\begin{remark}\label{flag hilbert schemes}
Let $i: Y \hookrightarrow \mathbb{P}^M$ be an embedding of a smooth variety inside a projective space. Let $D \in |L^{-1}|$ be a smooth divisor. Then there is an induced embedding of $j: D \hookrightarrow \mathbb{P}^M$. Let $p_1$ and $p_2$ denote the Hilbert polynomials of $Y \hookrightarrow \mathbb{P}^M$ and $D \hookrightarrow \mathbb{P}^M$ respectively. Let $\mathcal{H}(p_1, p_2)$ denote the flag Hilbert scheme representing the flag-Hilbert functor as defined in \cite{Ser}, Section $4.5.1$ or \cite{KL81}. Closed points of $\mathcal{H}(p_1, p_2)$ parameterize pairs $(D' \hookrightarrow Y' \hookrightarrow \mathbb{P}^M)$ where the subschemes $Y' \hookrightarrow \mathbb{P}^M$ and $D' \hookrightarrow \mathbb{P}^M$ have Hilbert polynomials $p_1$ and $p_2$ respectively. Let $\mathcal{D}$ denote the Hilbert scheme parameterizing subschemes $(D' \hookrightarrow \mathbb{P}^M)$ with Hilbert polynomial $p_2$. There is a projection map $p: \mathcal{H}(p_1, p_2) \to \mathcal{D}$ whose fibre $\mathcal{F}_D$ at a subscheme $(D \hookrightarrow \mathbb{P}^M)$ consists of all those subschemes $Y' \hookrightarrow \mathbb{P}^N$ with Hilbert polynomial $p_1$ such that $D \subset Y'$. Let $\mathcal F_{D,Y}$ denote the union of irreducible components of $\mathcal F_D$ containing the point $(D \hookrightarrow Y\hookrightarrow \mathbb{P}^M)$. Closed points parameterized by $\mathcal{F}_{D,Y}$ correspond to deformations over an irreducible curve of the subvariety $Y$ inside $\mathbb{P}^M$ which keeps $D$ fixed.     
\end{remark}

\begin{theorem}\label{Hirzebruch main}
Let $Y = \mathbb{F}_e \hookrightarrow \mathbb{P}^N$ be an embedding of the Hirzebruch surface $\mathbb{F}_e$, with $e \leq 2$, induced by the complete linear series of the very ample line bundle $|aC_0+bf|$ (hence $b \geq ae+1$). Let $i: \mathbb{P}^N \to \mathbb{P}^M$ be a linear embedding with $M = N + h^0((a-2)C_0+(b-e-2)f)$ and consider the composed (possibly degenerate) embedding $Y \hookrightarrow \mathbb{P}^M$. Let $E \in |-K_Y|$ denote a smooth elliptic curve. 

\begin{itemize}
    \item[(1)] If $Y_1$ is the subvariety $Y$ of $\mathbb P^M$ and $Y_2$ is a subvariety of $\mathbb P^M$ that corresponds to a general element of $\mathcal{F}_{E,Y}$, then {the scheme theoretic intersection of $Y_1 and Y_2$ is $E$ and} {$V = Y_1 \bigcup_E Y_2$} is  smoothable inside $\mathbb{P}^M$ to a \textit{smooth $K3$ surface} embedded by the complete linear series of a very ample line bundle. The Hilbert point of $V$ inside $\mathbb{P}^M$ is smooth. 
    %\begin{enumerate}
     %   \item[(a)] If $M = N$, the subvarieties $V$ and the $K3$ surfaces are embedded by a sublinear series of codimension $$h^0((a-2)C_0+(b-e-2)f) = \displaystyle\frac{(a-1)}{2}(2b-2-ae)$$ of a very ample line bundle.
      %  \item[(b)] If $M = N + h^0((a-2)C_0+(b-e-2)f)$, the subvarieties $V$ and the $K3$ surfaces are embedded by a complete linear series  of a very ample line bundle.
    %\end{enumerate}
     Further the closure inside the Hilbert scheme of the locus of varieties of the form $V = Y_1 \bigcup_E Y_2$ contains embedded $K3$ carpets $\widetilde{Y} \hookrightarrow \mathbb{P}^M$ which extends the embedding $Y \hookrightarrow \mathbb{P}^M$.  
    
    \smallskip
    
    \item[(2)] The locally trivial deformations of $V$ inside $\mathbb{P}^M$ are unobstructed and form a subspace in the tangent space of the Hilbert scheme at $[V]$ of codimension $h^0(-2K_Y|_E) = 6$. Furthermore, this subspace is the tangent space at $[V]$ of an irreducible locus of the Hilbert scheme, which has codimension $h^0(-2K_Y|_E) = 6$ and is smooth at $[V]$. The singularities of the subschemes parameterized by this locus are normal crossing singularities;  in particular, they are non-normal and semi-log-canonical. They are analytically isomorphic to $(x_1^2+x_2^2 = 0 \subset \mathbb{C}^M)$. 
   % \item[(3)] The general fiber of any one-parameter family of embedded deformations of $V$ whose image under the Kodaira-Spencer map is a  locally trivial deformation of $V$, is once again a simple normal crossing variety $V' = Y_1' \bigcup_D' Y_2'$, where both $Y_1'$ and $Y_2'$ are deformations of the subvariety $Y$ and $D'$ is a deformation of the subscheme $D$. 
    \smallskip
    
    \item[(3)] The general smooth surface of the Hilbert component containing $V$ is a 
    \begin{itemize}
        \item[(a)] prime $K3$ surface if $\textrm{gcd}(a,b) = 1$ 
        \item[(b)] a non-prime $K3$ surface embedded by $rB$, with $r = \operatorname{gcd}(a,b)$ and $B$ is a primitive very ample line bundle. 
    \end{itemize}
\end{itemize}

\end{theorem}

\begin{proof}
Let us prove part $(1)$. We apply \cite[Theorem $2.10$]{BGM24}. First of all, there exists nowhere vanishing sections inside $H^0(N_{Y/\mathbb{P}^M} \otimes K_Y)$ corresponding to embedded ribbons $\widetilde{Y}$ by \cite[Theorem $1.7$]{GP97} for $a = 1$ and  by \cite[Theorem $3.1$]{BMR21} for $a \geq 2$. A simple computation  shows that $H^1(N_{Y/\mathbb{P}^M}) = 0$ and hence $Y$ is unobstructed in $\mathbb{P}^M$. If $a = 1$, then we have that $H^1(N_{Y/\mathbb{P}^M} \otimes K_Y) = 0$ by \cite[Theorem $1.7$]{GP97} while if $a \geq 2$, the same vanishing follows from \cite[Lemma $2.7$, $(2)$]{BMR21} and hence, the functor $F_{E, Y}$ is unobstructed. If $a = 1$, the smoothability of $\widetilde{Y}$ inside $\mathbb{P}^M$ and the smoothness of its Hilbert point follows from \cite[Corollary $2.9$]{GGP13} and \cite[Theorem $4.1$]{GP97} respectively while if $a \geq 2$, the same follows from \cite[Theorem $4.1$]{BMR21} and \cite[Theorem $5.1$]{BMR21} respectively. 

Therefore part $(1)$ follows from Theorem \cite[Theorem $2.10$]{BGM24}. \\
To show part $(2)$ we need to show that $H^1(N_{V/\mathbb{P}^M}') = 0$ for one subvariety of the form $V = Y_1 \bigcup_E Y_2$ where $Y_2 \in \mathcal{F}_{E,Y}$. For this, first note that for any such $V$ we have an exact sequence 
$$0 \to N_{V/\mathbb{P}^M}' \to N_{V/\mathbb{P}^M} \to T_V^1 \to 0$$
Note that by part $(1)$, we have flat family $\mathcal{Y} \to T$ over an smooth irreducible curve $T$, such that $\mathcal{Y}_0$ is an embedded $K3$ carpet $\widetilde{Y}$ on $Y$ and $\mathcal{Y}_t$ is a subvariety of the form $V = Y_1 \bigcup_E Y_2$. When $a \geq 2$, by \cite[Theorem $5.1$, equation $(5.9)$]{BMR21}, we have that $H^1(N_{\widetilde{Y}/\mathbb{P}^M}) = 0$, while for $a = 1$, the same holds for by \cite[Theorem $4.1$]{GP97}. Now since the family $\mathcal{Y} \to T$ is local complete intersection, we have that $N_{\mathcal{Y}/\mathbb{P}_T^M}$ is locally free and hence flat over $T$. Therefore by semi-continuity, $H^1(N_{\mathcal{Y}_t/\mathbb{P}^M}) = 0$, i.e, $H^1(N_{V/\mathbb{P}^M}) = 0$. The vanishing of $H^1(N_{V/\mathbb{P}^M}')$ now follows if we show that $H^0(N_{V/\mathbb{P}^M}) \to H^0(T_V^1)$ is a surjection. Once again, we show this surjection by passing to the ribbon, i.e, we show that $H^0(N_{\widetilde{Y}/\mathbb{P}^M}) \to H^0(T_{\widetilde{Y}}^1)$ is surjection. First note that the sheaf $T_{\widetilde{Y}}^1 = \mathcal{O}_Y(-2K_Y)$ and sits as the last term in the exact sequence 
\begin{equation*}
    0 \to \mathcal{H}om(I_{\widetilde{Y}}/I_{Y}^2, \mathcal{O}_Y) \to N_{\widetilde{Y}/\mathbb{P}^M} \otimes \mathcal{O}_Y \to \mathcal{O}_Y(-2K_Y) \to 0
\end{equation*}

while the map $H^0(N_{\widetilde{Y}/\mathbb{P}^M}) \to H^0(T_{\widetilde{Y}}^1)$ factors as $H^0(N_{\widetilde{Y}/\mathbb{P}^M}) \to H^0(N_{\widetilde{Y}/\mathbb{P}^M} \otimes \mathcal{O}_Y) \to H^0(T_{\widetilde{Y}}^1)$ where the first map is obtained by taking the cohomology of the exact sequence

\begin{equation*}
 0 \to N_{\widetilde{Y}/\mathbb{P}^M} (K_Y) \to N_{\widetilde{Y}/\mathbb{P}^M} \to N_{\widetilde{Y}/\mathbb{P}^M} \otimes \mathcal{O}_Y \to 0   
\end{equation*}

So to show the surjection it is enough to that $H^1(\mathcal{H}om(I_{\widetilde{Y}}/I_{Y}^2, \mathcal{O}_Y)) = 0$ and $H^1(N_{\widetilde{Y}/\mathbb{P}^M} (K_Y)) = 0$. For $a = 1$, these vanishings follow from the proof of \cite[Theorem $4.1$]{GP97} and for $a \geq 2$, from \cite[Theorem $5.1$]{BMR21}.

%from the following two lemmas which are generalizations of \cite[Lemma $(2)$]{CLM93} and \cite[Corollary $(1)$]{CLM93} for any $a \geq 1$, whose proof works exactly as in \cite{CLM93} 
%(note that we have already showed in Theorem \ref{snc to carpet} that $H^1(N_{Y/\mathbb{P}^M}) = 0$)  
%\begin{lemma}\label{alternate T^1}
%The sheaf $T_V^1$ is the cokernel of the inclusion $N_{Y/\mathbb{P}^M} \to N_{V/\mathbb{P}^M}|_Y$
%\end{lemma}

%\begin{lemma}\label{alternate smoothing}
%The map of global sections from $H^0(N_{V/\mathbb{P}^M}) \to H^0(T_V^1)$ is a surjection. 
%\end{lemma}

%Part $(3)$ follows from \Cref{degeneration into $K3$ carpets}, part $(3)$.

\end{proof}

\color{black}

\vspace{0.5cm}

\bibliographystyle{plain}

\end{document}